\documentclass{article}
\usepackage{graphicx, amsmath, amssymb,float,scalerel,xcolor,dirtytalk, gensymb, amsthm, caption, comment, etoolbox, cite} % Required for inserting images
\usepackage[a4paper, total={5.5in, 8in}]{geometry}

\AtBeginEnvironment{align}{\setcounter{equation}{0}}

\DeclareMathOperator*{\concat}{\scalerel*{\Vert}{\sum}}

\newcommand{\mb}[1]{\mathbb{#1}}
\newcommand{\mc}[1]{\mathcal{#1}}

\theoremstyle{definition}
\newtheorem{definition}{Definition}

\theoremstyle{plain}
\newtheorem{lemma}{Lemma}

\theoremstyle{plain}
\newtheorem{proposition}{Proposition}

\theoremstyle{plain}
\newtheorem{theorem}{Theorem}

\theoremstyle{plain}
\newtheorem{corollary}{Corollary}[theorem]

\title{Generalized results on the convergence\\ of Thue-Morse turtle curves}
\author{Leif Schaumann}
\date{December 2024}

\begin{document}
\maketitle

\begin{abstract}
    Work by Ma and Holdener in 2005 \cite{ma_holdener} revealed that using turtle graphics to visualize the Thue-Morse sequence can result in curves which approximate the Koch fractal curve. A 2007 paper by Allouche and Skordev \cite{revisited} pointed out that this phenomenon is connected to certain complex sums considered by Dekking in 1982 \cite{dekking}. We make this connection explicit by showing that a broad class of Thue-Morse turtle curves will periodically coincide with sums of the form considered by Dekking, and we use this result to prove that scaled versions of these curves can converge to various fractal curves in the Hausdorff metric. In particular, we confirm a conjecture by Zantema \cite{zantema}, giving a condition under which any Thue-Morse turtle curve will converge to the Koch curve.
\end{abstract}

\section{Introduction}

The Thue-Morse sequence is one of the most well known morphic sequences; it appears in many places across mathematics (see \cite{ubiquitous} for examples). It is an infinite sequence of 0's and 1's which can be defined by a simple replacement rule: replace every 0 with 01 and every 1 with 10. Starting with a single 0, the limit of this replacement process produces the infinite Thue-Morse sequence
$$(t_2)=0,1,1,0,1,0,0,1,1,0,0,1,0,1,1,0,1,0,0,1,0,1,1,0,0,1,1,0...$$

A common way of visualizing such a sequence is through turtle graphics. We place a ``turtle'' in the plane at the origin, facing east. We then give the turtle a sequence of instructions, each of which consists of a translation of the turtle's position and a rotation of the turtle's heading. For each instruction, the translation is applied before the rotation, and the translation is relative to the turtle's current heading. We can visualize the turtle's movements as a series of connected line segments in the plane called a \textit{turtle curve}. 

If we want to use this method to visualize the Thue-Morse sequence, we just need to pick two instructions for the turtle to execute according to the appearance of 0's and 1's in the sequence. We call a turtle curve produced in this way a \textit{Thue-Morse turtle curve}. For example, say that every time a 0 appears in the Thue-Morse sequence, we have the turtle move forward by 1 and then rotate by $\frac{\pi}{3}$, and every time a 1 appears in the Thue-Morse sequence, we have the turtle move forward by 1 and then rotate by $-\frac{\pi}{3}$. This assignment of instructions produces the turtle curve shown in Figure \ref{fig:simple_tmtc}.

\begin{figure}[H]
    \centering
    \includegraphics[scale=0.6]{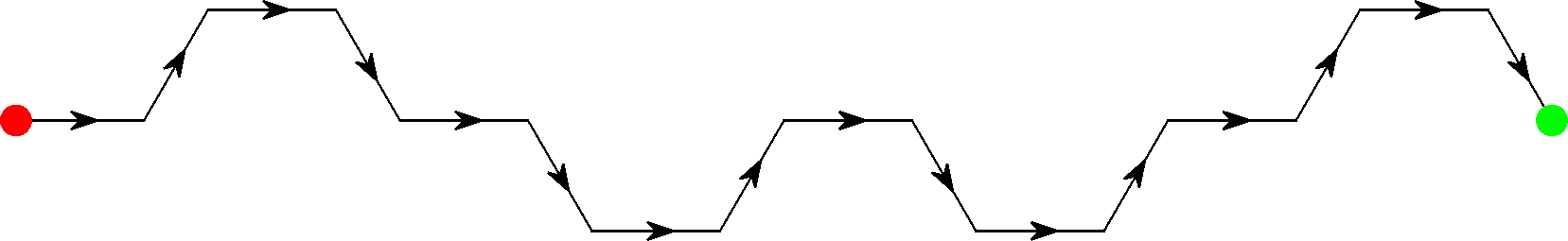}
    \caption{The first 16 steps of a Thue-Morse turtle curve where the turtle turns by $\frac{\pi}{3}$ when reading a 0 and by $-\frac{\pi}{3}$ when reading a 1. The origin is marked in red and the current position of the turtle is marked in green.}
    \label{fig:simple_tmtc}
\end{figure}

The Thue-Morse turtle curve in Figure \ref{fig:simple_tmtc} is not particularly remarkable; it continues forever in an essentially linear path. However, a wide range of other behaviors can be achieved depending on the assignment of instructions. A 2005 paper by Ma and Holdener \cite{ma_holdener} revealed that scaled iterations of certain Thue-Morse turtle curves converge to the famous Von Koch fractal curve in the Hausdorff metric, as seen in Figure \ref{fig:koch_convergence}. This was a surprising connection between two well known mathematical objects, making it worthy of further investigation.

\begin{figure}[H]
    \centering
    \includegraphics[scale=0.5]{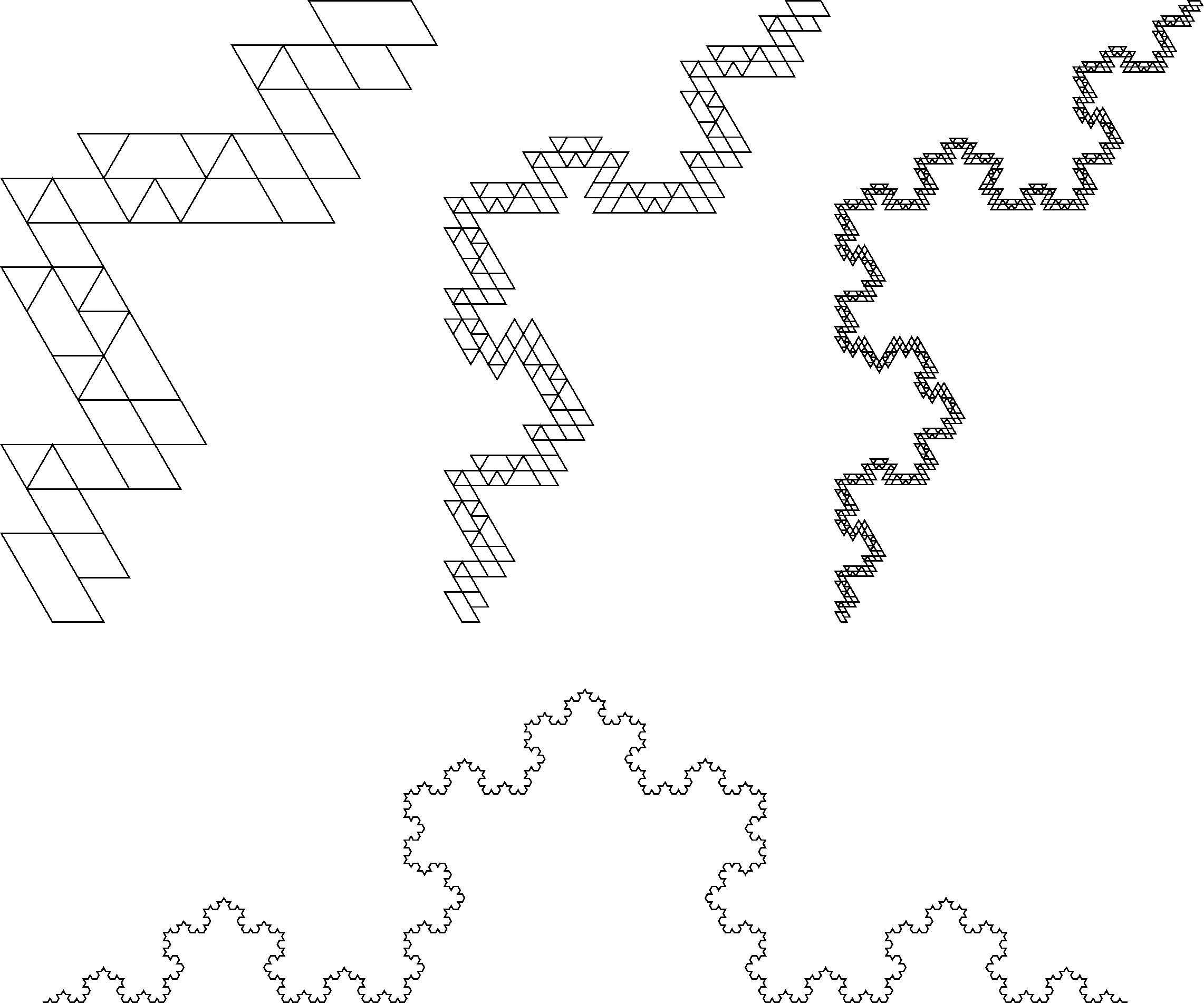}
    \caption{Three scaled iterations of the Thue-Morse turtle curve considered by Ma and Holdener. In this turtle curve, the turtle moves forward 1 unit whenever reading a 0, and turns left by $\frac{\pi}{3}$ whenever reading a 1. A classic approximation of the Koch curve is shown below.}
    \label{fig:koch_convergence}
\end{figure}

A 2007 paper by Allouche and Skordev pointed out that the connection between Thue-Morse sequence and the Koch curve was closely related to a paper by Dekking from 1982 \cite{revisited}. Dekking's paper considered sums of complex numbers of the form 
$$Z(N,p,q)=\sum_{n=0}^{N-1}\zeta_p^{s_p(n)}\zeta_q^n,$$
where $p$ and $q$ are integers with $p,q\geq2$, $\zeta_p=e^{\frac{2\pi i}{p}}$, $\zeta_q=e^{\frac{2\pi i}{q}}$, and $s_p(n)$ is the sum of the digits of $n$ when written in base $p$. Allouche and Skordev gave an overview of how these sums are relevant to the Thue-Morse sequence and Koch curve connection, but they did not provide explicit results or fully explore possible generalizations. In this paper, we aim to fill in those details and investigate the extent of this connection. We will find explicit relations between Dekking's sums and Thue-Morse turtle curves, and use these relations to prove the convergence of various Thue-Morse turtle curves to fractal curves such as the Koch curve. 

\section{Preliminaries}

We start by recalling some concepts from combinatorics on words, which will be used to define the Thue-Morse sequence and other relevant sequences. We will then establish a flexible framework for defining turtle curves.

\subsection{Words and Sequences}

For a finite alphabet $A$, we denote $A^*$ to be the set of all finite words over $A$, and $\epsilon$ to be the empty word. The concatenation of the words $u$ and $v$ will be written $uv$ or $u\concat v$. Also, for any word $w\in A^*$, let $|w|$ denote the length of $w$, and let $|w|_a$ denote the number of occurrences of the symbol $a$ in the word $w$. A \textit{sequence} $\sigma$ over the alphabet $A$ is any function $\sigma:\mb{N}\to A$. For our purposes, $\mb{N}$ includes 0.

In order to extract a symbol from a word at a certain position, we will use a notation inspired by computer programming syntax. We denote $w[n]$ to be the symbol in the word $w$ at position $n$. The first symbol in the word is considered to be in position 0.

To extract a word from a sequence $\sigma$, we introduce the following. For any $a,b\in\mb{N}$, $a\leq b$, let $\sigma[a:b]$ denote the word formed with the symbols of $\sigma$ starting at index $a$ and ending at index $b$, including $\sigma(a)$ and excluding $\sigma(b)$. More formally, $\sigma[a:b]=\concat_{j=a}^{b-1}\sigma(j)$, where $\concat$ is the concatenation operation. If $a=b$, $\sigma[a:b]=\epsilon$. We most commonly have $a=0$, so we will use $\sigma[:b]$ as shorthand for $\sigma[0:b]$.

A \textit{morphism} is any function $\phi:A^*\to A^*$ with $\phi(uv)=\phi(u)\phi(v)$ for all $u,v\in A^*$. Note that with this property, describing where $\phi$ sends each symbol in $A^*$ is enough to fully define the morphism. If there exists some symbol $a\in A$ such that $\phi(a)=aw$ for some word $w$, then we say $\phi$ is \textit{prolongable} on $a$, and the infinite word $\phi^\omega(a)=aw\phi(w)\phi^2(w)\phi^3(w)\phi^4(w)...$ is the unique fixed point of $\phi$ starting with $a$. We call a morphism $\phi$ \textit{k-uniform} if $|\phi(a)|=k$ for all $a\in A$ for some $k\in\mb{N}$  \cite{auto_sequences}. For an infinite word $\phi^\omega(a)$ and sequence $\sigma$, we will abuse notation slightly by using $\phi^\omega(a)=\sigma$ to mean that $\phi^\omega(a)[n]=\sigma(n)$ for all $n\in\mb{N}$.

The following lemma will be useful in many of our calculations. It allows us to think of the fixed point of a $k$-uniform morphism as a sequence with a recurrence relation.

\begin{lemma}\label{lem:morphism_is_recurrence}
Let $A$ be a finite alphabet and $\phi:A^*\to A^*$ be a $k$-uniform morphism with $k\geq2$ which is prolongable on some $a\in A$.  Then for all $m,r\in\mb{Z}$ with $0\leq r<k$, $$\phi^\omega(a)[mk+r]=\phi(\phi^\omega(a)[m])[r].$$
\end{lemma}

\begin{proof}
    First note a general fact about $k$-uniform morphisms, which is that for any $n\in\mb{N}$, $\phi(\phi^\omega(a)[:n])=\phi^\omega(a)[:nk]$. This is because $\phi^\omega(a)$ is a fixed point of $\phi$, and since $\phi$ is $k$-uniform, applying $\phi$ to a word of length $n$ gives a word of length $nk$. We will apply this fact twice in the following calculation.
    \begin{align}
        \phi^\omega(a)[mk+r]&=\phi^\omega(a)[:mk+k][mk+r]\\
        &=\phi(\phi^\omega(a)[:m+1])[mk+r]\\
        &=\phi(\phi^\omega(a)[:m]\phi^\omega(a)[m])[mk+r]\\
        &=\phi(\phi^\omega(a)[:m])\phi(\phi^\omega(a)[m])[mk+r]\\
        &=\phi^\omega(a)[:mk]\phi(\phi^\omega(a)[m])[mk+r]\\
        &=\phi(\phi^\omega(a)[m])[r]
    \end{align}
    The first equality holds because $r<k$, so $0\leq mk+r<mk+k$. The second and fifth equalities follow from the fact noted above. The final equality is simply removing the first $mk$ symbols from the word, and shifting the indexing value by $mk$ to account for this. 
\end{proof}

\subsection{The Thue-Morse Sequence}

We first define the Thue-Morse sequence by a recurrence relation.

\begin{definition}
    The \textit{Thue-Morse sequence} is the sequence $t_2:\mb{N}\to\mb{Z}/2\mb{Z}$ is defined as follows. Let $t_2(0)=0$. For any $m\in\mb{N}$ and integer $0\leq r<2$, let $t_2(2m+r)=t_2(m)+\overline{r}$, where $\overline{r}$ is the equivalence class of $r$ modulo 2.
\end{definition}

A common, equivalent definition of the Thue-Morse sequence is as the fixed point of a 2-uniform morphism.

\begin{proposition}\label{prop:t_2_equivalent_defs}
    For the alphabet $A_2=\mb{Z}/2\mb{Z}$, let $\phi:A^*\to A^*$ be the morphism defined by $\phi(0)=01$ and $\phi(1)=01$. Then $\phi^\omega(0)=t_2$, where $t_2$ is the \textit{Thue-Morse sequence}.
\end{proposition}

We save the proof of Proposition \ref{prop:t_2_equivalent_defs} for the more general case covered by Proposition \ref{prop:t_p_equivalent_defs}, but a proof can also be found in \cite{ubiquitous}. Both definitions of the Thue-Morse sequence yield $$(t_2)=0,1,1,0,1,0,0,1,1,0,0,1,0,1,1,0,1,0,0,1,0,1,1,0,0,1,1,0...$$

We will also consider a generalization of the Thue-Morse sequence which allows for sequences using more than 2 symbols. Although our main result will only apply to turtle curves generated by the 2-symbol Thue-Morse sequence, many of our intermediate results will still apply with this generalization.

\begin{definition}
    For any $p\geq2$, the \textit{generalized Thue-Morse sequence over $p$ symbols} $t_p:\mb{N}\to\mb{Z}/p\mb{Z}$ is defined as follows. Let $t_p(0)=0$. For any $m\in\mb{N}$ and integer $0\leq r<p$, let $t_p(pm+r)=t_p(m)+\overline{r}$, where $\overline{r}$ is the equivalence class of $r$ modulo $p$.
\end{definition}

Note that this is equivalent to Definition 1 when $p=2$. Here are some examples of generalized Thue-Morse sequences with $p=3$, $p=4$, and $p=5$.
\begin{align*}
    (t_3)&=0,1,2,1,2,0,2,0,1,1,2,0,2,0,1,0,1,2,2,0,1,0,1,2,1,2,0...\\
    (t_4)&=0,1,2,3,1,2,3,0,2,3,0,1,3,0,1,2,1,2,3,0,2,3,0,1,3,0,1...\\
    (t_5)&=0,1,2,3,4,1,2,3,4,0,2,3,4,0,1,3,4,0,1,2,4,0,1,2,3,1,2...
\end{align*}
Like the 2-symbol Thue-Morse sequence, the generalized Thue-Morse sequences can be defined equivalently as the fixed point of a morphism.

\begin{proposition}\label{prop:t_p_equivalent_defs}
    Let $p$ be an integer with $p\geq2$. Treat $A_p=\mb{Z}/p\mb{Z}$ as an alphabet. Let $\phi:A_p^*\to A_p^*$ be the morphism defined by $\phi(a)=\concat_{j=0}^{p-1}(a+\overline{j})$. Then $\phi^\omega(0)=t_p$, where $t_p$ is the generalized Thue-Morse sequence over $p$ symbols.
\end{proposition}

\begin{proof}
    First observe that $\phi^\omega(0)[0]=0=t_p(0)$, so $\phi^\omega(0)$ and $t_p$ start with the same symbol. We will show that  $\phi^\omega(0)$ satisfies the same recurrence relation as $t_p$. Let $m,r\in\mb{N}$ with $0\leq r<p$. Since $\phi$ is a $p$-uniform morphism prolongable on 0, we can apply Lemma \ref{lem:morphism_is_recurrence} to calculate $\phi^\omega(0)[mp+r]$, giving
    $$\phi^\omega(0)[mp+r]=\phi(\phi^\omega(0)[m])[r]=\phi^\omega(0)[m]+\overline{r}$$
    The second equality is by the definition of $\phi$ and the fact that $0\leq r<p$. Thus, we have shown that the sequence $\phi^\omega(0)$ follows the same recurrence relation as $t_p$. Since we also have $\phi^\omega(0)[0]=t_p(0)$, we must have $\phi^\omega(0)=t_p$.
\end{proof}

We will now demonstrate a fact about Thue-Morse sequences which will be important in the proofs of the coming theorems.

\begin{lemma}\label{lem:p^k_reduction}
    Let $t_p$ be the generalized Thue-Morse sequence over $p$ symbols. Then for all $m,k,r\in\mb{N}$ with $k\neq0$ and $0\leq r<p^k$, $$t_p(mp^k+r)=t_p(m)+t_p(r).$$
\end{lemma}

\begin{proof}
    We will induct on $k$. In the base case $k=1$, we can just apply the definition of $t_p$ to find that $t_p(mp+r)=t_p(m)+\overline{r}$. Since $0\leq r<p$, we have $t_p(r)=\overline{r}$, so $t_p(mp+r)=t_p(m)+t_p(r)$, settling the base case.

Now assume that for some arbitrary $k\in\mb{N}$, $t_p(mp^k+r)\equiv t_p(m)+t_p(r)$ holds for all integers $m$ and $0\leq r<p^k$. Let $m_1,r_1\in\mb{N}$ with $0\leq r_1<p^{k+1}$. We then write $r_1=m_2p^k+r_2$, with $0\leq m_2<p$ and $0\leq r_2<p^k$. We can now compute $t_p(m_1p^{k+1}+r_1)$.
\begin{align*}
t_p(m_1p^{k+1}+r_1)&=t_p(m_1p^{k+1}+m_2p^k+r_2)\\
&=t_p((m_1p+m_2)p^k+r_2)\\
&=t_p(m_1p+m_2)+t_p(r_2)\\
&=t_p(m_1)+t_p(m_2)+t_p(r_2)\\
&=t_p(m_1)+t_p(r_1)
\end{align*}

We find the third congruence by applying the induction hypothesis, the fourth by the base case, and last because $t_p(r_1)=t_p(m_2p^k+r_2)=t_p(m_2)+t_p(r_2)$. This completes the induction.
\end{proof}

\subsection{Turtle Curves}

We will now formalize the concept of a turtle curve as described in the introduction. A turtle's state consists of two pieces of information: a position and a heading. Each instruction that the turtle follows also consists of two pieces of information: a translation and a rotation. These two concepts really contain the same type of information, so we will define a structure which can represent both turtle states and turtle instructions.

\begin{definition}
    Let $\mc{G}=\mb{C}\times\mb{S}^1$, where $\mb{S}^1=\{z\in\mb{C}:|z|=1\}$. Define $+:\mc{G}\times \mc{G}\to \mc{G}$ by letting
    $$(z_1,u_1)+(z_2,u_2)=(z_1+u_1z_2, u_1u_2)$$
    for any $(z_1,u_1),(z_2,u_2)\in \mc{G}$. We call $(\mc{G},+)$ the \textit{turtle group}.
\end{definition}

An element $(z,u)\in\mc{G}$ can be interpreted in two different ways. The first is to think of $(z,u)$ as the state of a turtle; $z$ is the position of the turtle, and $u$ is a complex number on the unit circle which represents the heading of the turtle. The second way is to think of $(z,u)$ as an instruction given to the turtle: ``first translate by $z$, and then adjust your heading by multiplying it by $u$.'' The sum of two elements $g_1,g_2\in\mc{G}$ can be thought of as the resulting turtle state when a turtle starts with the state $g_1$ and performs the instruction $g_2$. 

Note that if a turtle starts in the state $(0,1)$ and performs an instruction $(z,u)\in\mc{G}$, the resulting state will be $(0,1)+(z,u)=(0+1\cdot z,1\cdot u)=(z,u)$. Because of this, the two interpretations of $(z,u)$ as a state and as an instruction are compatible, and $(0,1)$ can be thought of as an identity element of $\mc{G}$. In fact, as hinted by the name, the turtle group $\mc{G}$ is indeed a group\footnote{This is the motivation for having the turtle rotate \textit{after} translating. If the operation on $\mc{G}$ was defined so that the turtle rotated before moving, then $\mc{G}$ would no longer be a group.}.

\begin{proposition}
    $(\mc{G},+)$ is a group.
\end{proposition}

\begin{proof}
    As mentioned above, it is not difficult to check that $(0,1)$ is an identity element of $\mc{G}$. For any element $(z,u)\in\mc{G}$, its inverse is $(-\frac{z}{u},\frac{1}{u})\in\mc{G}$. To check associativity, let $(z_1,u_1),(z_2,u_2),(z_3,u_3)\in\mc{G}$. Then
    \begin{align*}
        ((z_1,u_1)+(z_2,u_2))+(z_3,u_3)&=(z_1+u_1z_2, u_1u_2)+(z_3,u_3)\\
        &=(z_1+u_1z_2+u_1u_2z_3,u_1u_2u_3)\\
        &=(z_1,u_1)+(z_2+u_2z_3,u_2u_3)\\
        &=(z_1,u_1)+((z_2,u_2)+(z_3,u_3)).
    \end{align*}
    Therefore $(\mc{G},+)$ is a group.
\end{proof}

Instead of directly producing a sequence of instructions in $\mc{G}$ for our turtle, we find it helpful to obtain our sequence of instructions from a sequence over some finite alphabet $A$. Thus, we define a function which can take a sequence in $A$ and map it to a sequence in $\mc{G}$.

\begin{definition}
    Given a finite alphabet $A$, an \textit{interpreter function} on $A$ is any function $\tau:A\to \mc{G}$. We extend this $\tau$ to functions $\tau:A^*\to\mc{G}^*$ and $\tau:A^\mb{N}\to\mc{G}^\mb{N}$ in the natural way.
\end{definition}

It will also be helpful to define some basic projection functions from $\mc{G}$ to $\mb{C}$ and $\mb{S}^1$.

\begin{definition}
    Define $\pi_1:\mc{G}\to\mb{C}$ by letting $\pi_1((z,u))=z$ for any $(z,u)\in \mc{G}$. Define $\pi_2:\mc{G}\to\mb{S}^1$ by letting $\pi_1((z,u))=u$ for any $(z,u)\in \mc{G}$.
\end{definition}

Now that we have defined how turtle instructions work and set up a way to produce them from a sequence of symbols, we are ready to formally define our concept of a turtle curve.

\begin{definition}
    Let $\sigma$ be a sequence over a finite alphabet $A$, and let $\tau$ be an interpreter function on $A$. We then call the tuple $T=(\sigma,\tau)$ a \textit{turtle curve}. We abuse notation by also viewing $T$ as a sequence over $\mb{C}$, letting
    $$T(n)=\pi_1\left(\sum_{i=0}^{n-1}\tau(\sigma(i))\right)$$
    for all $n\in\mb{N}$.
\end{definition}

In other words, $T(n)$ is the position of the turtle after executing the instructions represented by the first $n$ terms of the sequence $\sigma$. We now define several other functions to more easily access relevant information pertaining to a turtle curve.

\begin{definition}\label{def:SPalpha}
    Define $S:\mc{G}^*\to\mc{G}$ by letting 
    $$S(g_1g_2...g_n)=\sum_{i=1}^ng_i$$
    for any $g_1g_2...g_n\in \mc{G}^*$. Define $P=\pi_1\circ S$ and $\alpha=\pi_2\circ S$. In the context where we have a turtle curve $T$ with an interpreter function $\tau$, we let $S_T=S\circ\tau$, $P_T=P\circ\tau$, and $\alpha_T=\alpha\circ\tau$.
\end{definition}

The function $P$ can be thought of as a ``position function,'' taking in a word of turtle moves and returning position obtained after executing those moves. The function $\alpha$ can be thought of as a ``heading function,'' taking in a word of turtle moves and returning the heading obtained after executing those moves. The functions $P_T$ and $\alpha_T$ are similar except that they take in symbols in the alphabet of a turtle curve and automatically interpret them by applying $\tau$. While the value of $P_T$ can be difficult to calculate, the following proposition makes calculating $\alpha_T$ fairly straightforward.

\begin{proposition}\label{prop:alpha_is_morphism}
    The map $\alpha$ is a homomorphism from $(G^*,||)$ to $(\mb{S}^1,\cdot)$.
\end{proposition}

\begin{proof}
    First see that the map $S$ is clearly a homomorphism from $(G^*,||)$ to $(\mc{G},+)$ by definition. The projection map $\pi_2$ is a homomorphism because
    \begin{align*}
        \pi_2((z_1,u_1)+(z_2,u_2))&=\pi_2((z_1+u_1z_2, u_1u_2))\\
        &=u_1u_2\\
        &=\pi_2((z_1,u_1))\pi_2((z_2,u_2))
    \end{align*}
    for any $(z_1,u_1),(z_2,u_2)\in\mc{G}$. Thus, $\alpha=\pi_2\circ S$ is a homomorphism because it is a composition of homomorphisms.
\end{proof}

Stated more explicitly, Proposition \ref{prop:alpha_is_morphism} means that for any $g_1g_2...g_n\in\mc{G}^*$, we have
$$\alpha(g_1g_2...g_n)=\prod_{i=1}^n\pi_2(g_i).$$
The following lemma extends this result by giving us a way to calculate $P_T$ as well. This is similar to Lemma 2 in \cite{zantema}.

\begin{lemma}\label{lem:expand_S}
    Let $T$ be a turtle curve using a finite alphabet $A$, and let $w_1,w_2,...,w_n\in A^*$. Then
    $$S_T(w_1w_2...w_n)=\left(\sum_{i=1}^n\alpha_T(w_1w_2...w_{i-1})P_T(w_i),\prod_{i=1}^n\alpha_T(w_i)\right).$$
\end{lemma}

\begin{proof}
    We induct on $n$. As a base case, when $n=1$, we have
    $$S_T(w_1)=(P_T(w_1),\alpha_T(w_1))=(\alpha_T(\epsilon)P_T(w_1),\alpha_T(w_1)).$$
    Now assume that the statement is true for some $n\in\mb{N}$. Let $w_1,w_2,...,w_{n+1}\in A^*$. Then using the fact that $S_T$ and $\alpha_T$ are homomorphisms, we find
    \begin{align*}
        &S_T(w_1,w_2,...,w_{n+1})\\
        =&S_T(w_1,w_2,...,w_n)+S_T(w_{n+1})\\
        =&\left(\sum_{i=1}^n\alpha_T(w_1w_2...w_{i-1})P_T(w_i),\prod_{i=1}^n\alpha_T(w_i)\right)+(P_T(w_{n+1}),\alpha_T(w_{n+1}))\\
        =&\left(\sum_{i=1}^n\alpha_T(w_1w_2...w_{i-1})P_T(w_i)+P_T(w_{n+1})\prod_{i=1}^n\alpha_T(w_i),\alpha_T(w_{n+1})\prod_{i=1}^n\alpha_T(w_i)\right)\\
        =&\left(\sum_{i=1}^{n+1}\alpha_T(w_1w_2...w_{i-1})P_T(w_i),\prod_{i=1}^{n+1}\alpha_T(w_i)\right).
    \end{align*}
    This completes the induction step.
    
\end{proof}

While we have currently defined a turtle curve $T$ as a sequence of points in the complex plane, if we would like something closer to an actual curve, we can construct a union of line segments between each consecutive pair of points of $T$.

\begin{definition}
    For any $x,y\in\mb{C}$, define $\ell(x,y)=\{x+t(y-x):t\in[0,1]\}$.
    Then for any $T\in\mb{C}^\mb{N}$, let
    $$\mc{P}(T)=\bigcup_{i=0}^\infty\ell(T(i),T(i+1)).$$
    Similarly, for any word $w\in\mb{C}^*$ of length $n$, let
    $$\mc{P}(w)=\bigcup_{i=0}^{n-2}\ell(w[i],w[i+1]).$$
\end{definition}

Our later results on convergence will involve scaled copies of sets of the form $\mc{P}(T[:n])$.

\subsection{Thue-Morse Turtle Curves}

We now define our main object of interest, which is simply a turtle curve generated by a Thue-Morse sequence.

\begin{definition}
    A \textit{Thue-Morse turtle curve} is any turtle curve $T=(\sigma,\tau)$ with $\sigma=t_p$ for some generalized Thue-Morse sequence $t_p$.
\end{definition}

For example, the turtle curve considered by Ma and Holdener in the introduction of \cite{ma_holdener} is $T=(t_2,\tau)$, with $\tau(0)=(1,1)$ and $\tau(1)=(0,\zeta_6)$, where $\zeta_6=e^{\frac{2\pi i}{6}}$. In this turtle program, the turtle moves forward one unit whenever reading a 0 in its instructions, and turns left by $\frac{2\pi}{6}$ whenever reading a 1. Zantema showed an exact relationship between the Koch curve and the turtle curve $T=(t_2,\tau)$, with $\tau(0)=(1,\zeta_3)$ and $\tau(1)=(1,-1)$ \cite{zantema}. Both curves are shown in Figure \ref{fig:holdener_zantema_curves}.

\begin{figure}[H]
    \centering
    \includegraphics[scale=0.4]{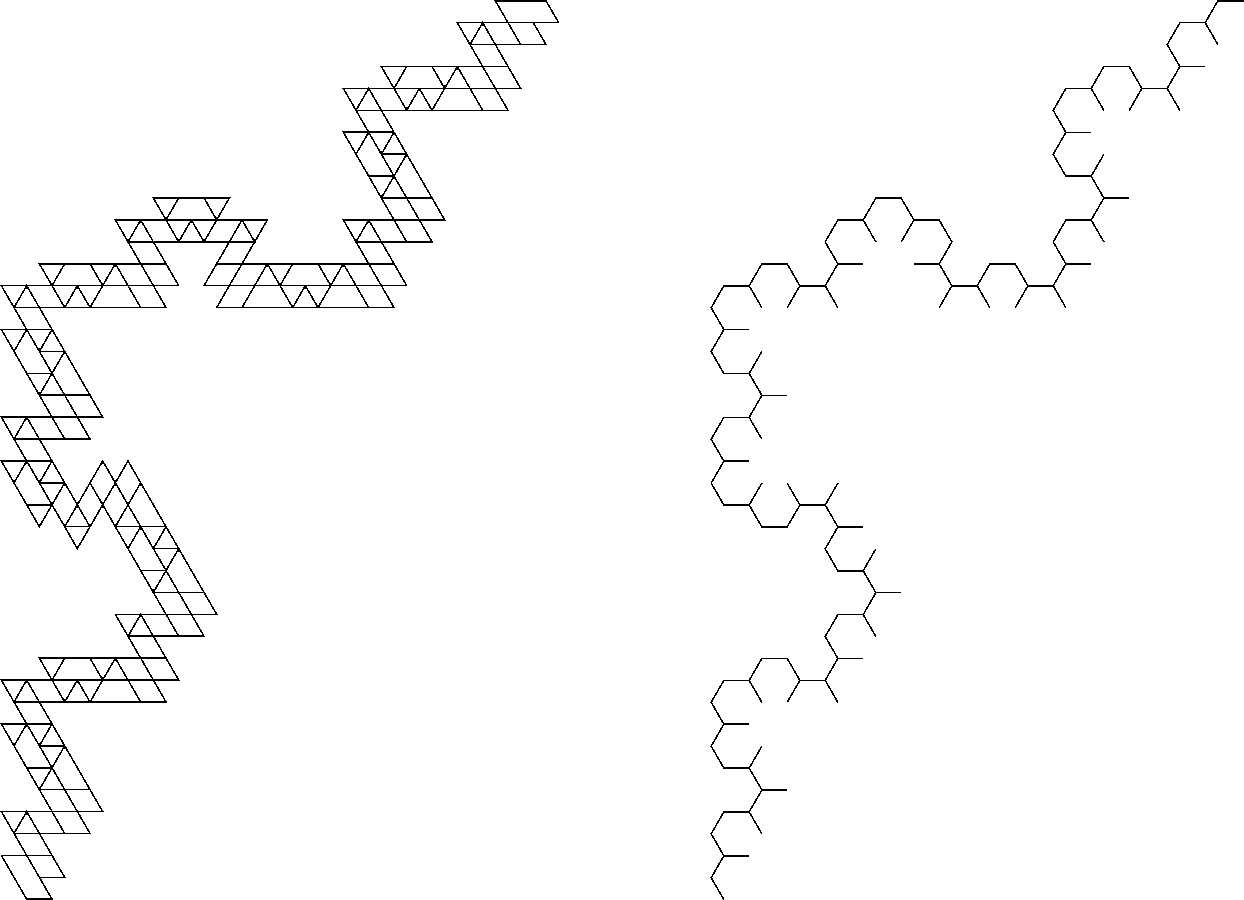}
    \caption{The curves considered by Ma and Holdener (left) and Zantema (right).}
    \label{fig:holdener_zantema_curves}
\end{figure}

\subsection{Absolute Turtle Curves}

One difficulty when analyzing turtle curves is the fact that the turtle group $\mc{G}$ is not commutative. When proving our results, we will work around this by relating our turtle curve of interest to a turtle curve in which the turtle's heading never changes. Such turtle curves are essentially just a sequence of partial sums of complex numbers, as demonstrated by the following proposition.

\begin{proposition}\label{prop:A_cong_C}
    The set $\mc{A}=\{g\in\mc{G}:\pi_2(g)=1\}$ is a subgroup of $\mc{G}$, and $\mc{A}\cong\mb{C}$.
\end{proposition}

\begin{proof}
    It is straightforward to check that the map $\pi_1:\mc{A}\to\mb{C}$ is an isomorphism.
\end{proof}

We now introduce language for turtle curves that live within $\mc{A}$.

\begin{definition}
    An \textit{absolute turtle curve} $T=(\sigma,\tau)$ is a turtle curve where $\tau(a)\in\mc{A}$ for all $a\in A$. In this context, because $\mc{A}\cong\mb{C}$, we often view $\tau$ as a function $\tau:A\to\mb{C}$.
\end{definition}

It is also worth noting the following fact.

\begin{proposition}
    If $T$ is an absolute turtle curve, then $P_T$ is a homomorphism from $(A,||)$ to $(\mb{C},+)$.
\end{proposition}

\begin{proof}
    Because $\pi_1$, $S$, and $\tau$ are all homomorphisms, $P_T=\pi_1\circ S\circ\tau$ is a homomorphism.
\end{proof}

\section{Connection to Dekking's Work}

In 2007, Allouche and Skordev \cite{revisited} pointed out that the paper of Ma and Holdener \cite{ma_holdener} was actually not the first time that a connection between the Thue-Morse sequence and the Koch curve had appeared. A 1982 paper by Dekking \cite{dekking} considered sums of the form 

$$Z(N,p,q)=\sum_{n=0}^{N-1}\zeta_p^{t_p(n)}\zeta_q^n,$$

where $p$ and $q$ are integers with $p,q\geq2$, $\zeta_p=e^{\frac{2\pi i}{p}}$, $\zeta_q=e^{\frac{2\pi i}{q}}$, and $t_p(n)$ is a generalized Thue-Morse sequence\footnote{Dekking actually used the function $s_p(n)$, the sum of the digits of $n$ when written in base $p$. However, as mentioned in \cite{ubiquitous}, $s_p(n)\equiv t_p(n)\pmod{p}$, which implies that $\zeta_p^{s_p(n)}=\zeta_p^{t_p(n)}$ for all $n\in\mb{N}$}. In the case where $p=2$ and $q=3$, drawing a line between $Z(N,2,3)$ and $Z(N+1,2,3)$ for every $N$ produces the classic approximation of the Koch curve using line segments shown in Figure \ref{fig:classic_koch}. To get an idea as to how this is relevant to turtle curves, compare the sums $Z(N,2,3)$ to the Thue-Morse turtle curve $T=(t_2,\tau)$ with $\tau(0)=(1,\zeta_3)$ and $\tau(1)=(-1,\zeta_3)$. $Z(N,2,3)$ gives the sums

$$Z(N,2,3)=\sum_{n=0}^{N-1}(-1)^{t_p(n)}\zeta_3^n.$$

Observe that in the Thue-Morse curve $T$, the turtle's heading will always rotate by an angle of $\frac{2\pi}{3}$ no matter if the next symbol is 0 or 1. This produces a period 3 sequence of constantly rotated headings over each step, which corresponds to the $\zeta_3^n$ factor in $Z(N,2,3)$. The $(-1)^{t_p(k)}$ factor in $Z(N,2,3)$ reflects how $P(0)=1=(-1)^0$ and $P(1)=-1=(-1)^1$. So $Z(N,2,3)$ and the $T$ are essentially the same, meaning that $\mc{P}(T)$ is also the classic approximation of the Koch curve using line segments shown in Figure \ref{fig:classic_koch}.

\begin{figure}[H]
    \centering
    \includegraphics[scale=0.7]{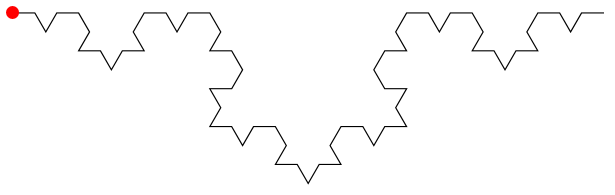}
    \caption{Both the turtle curve $T$ defined above and the sums $Z(N,2,3)$ produce the classic approximation of the Koch curve using line segments. The origin is marked with a red dot.}
    \label{fig:classic_koch}
\end{figure}

This is of course a very specific example, but our main results will show how a broader class of Thue-Morse turtle curves can be related to $Z(N,p,q)$.

\subsection{Dekking Sequences}

There is another, more direct way of interpreting $Z(N,p,q)$ as a turtle curve, which is to think of it as an absolute turtle curve. Let us imagine the sequence that such a turtle curve might use. Again, it is helpful to use $p=2$ and $q=3$ as a familiar example. Note that each term of the sum $Z(N,2,3)$ is one of six possible complex numbers, being the 6th roots of unity. This means that we can define a sequence over six symbols which encodes each possible term of the sum as a symbol. The determining factor of the $n$th term is the sum is the value of $t_2(n)$ and $n\mod3$, so those values will describe our symbols. This leads to the following definition for a general $p$ and $q$.

\begin{definition} 
    For any integers $p,q\geq2$, define $z_{p,q}:\mb{N}\to\mb{Z}/p\mb{Z}\times\mb{Z}/q\mb{Z}$ by $z_{p,q}(n)=(t_p(n),\overline{n}))$, where $\overline{n}$ is the congruence class of $n$ modulo $q$.  We call $z_{p,q}$ a \textit{Dekking sequence}.
\end{definition}

Here are some examples of Dekking sequences.
\begin{align*}
    (z_{2,3})&=(0,0),(1,1),(1,2),(0,0),(1,1),(0,2),(0,0),(1,1),(1,2),(0,0),(0,1),(1,2)...\\
    (z_{2,5})&=(0,0),(1,1),(1,2),(0,3),(1,4),(0,0),(0,1),(1,2),(1,3),(0,4),(0,0),(1,1)...\\
    (z_{3,2})&=(0,0),(1,1),(2,0),(1,1),(2,0),(0,1),(2,0),(0,1),(1,0),(1,1),(2,0),(0,1)...
\end{align*}

So we are essentially just taking a Thue-Morse sequence and putting it side by side with the simple periodic sequence $f_q(n)=n\mod q$. By thinking of these two sequences as one, we have a sequence which captures the interplay between them. 

The important fact here is that when $\gcd(p,q)=1$, the Dekking sequence $z_{p,q}$ can be defined as the fixed point of a morphism. To show how, we will use a trick: for some $k$, we will find a $k$-uniform morphism which generates $t_p$ and a $k$-uniform morphism which generates $f_q$. Since these morphisms have the same length, we can then combine them to form a $k$-uniform morphism which generates $z_{p,q}$.

In the following propositions, let $p,q$ be integers with $p,q\geq2$ and $\gcd(p,q)=1$.  Let $Q=p^{\varphi(q)}$, where $\varphi$ is the Euler totient function. We start by expressing the periodic sequence $f_q$ as the fixed point of a $Q$-uniform morphism.

\begin{proposition}\label{prop:delta_is_uniform}
     Treat $A_q=\mb{Z}/q\mb{Z}$ as an alphabet. Define the morphism $\delta$ on $A_q^*$ by $\delta(a)=\concat_{j=0}^{Q-1}(a+f_q(j))$ for all $a\in A_q$. Then $\delta$ is a $Q$-uniform morphism with $\delta^\omega(0)=f_q$.
\end{proposition}

\begin{proof}
    We will induct on $k$ in the following statement: $\delta^\omega(0)[n]=f_q(n)$ for all $0\leq n<Q^k$. This is true in the base case of $k=0$, as $\delta^\omega(0)[0]=0=f_q(0)$.
    
    Now assume that the statement holds for some arbitrary $k\in\mb{N}$. Let $n\in\mb{N}$ with $n<Q^{k+1}$. Then we can write $n=mQ+r$ for some $m,r\in\mb{N}$ with $m<Q^k$ and $r<Q$. $\delta$ is a $Q$-uniform morphism prolongable on 0, so we can use Lemma \ref{lem:morphism_is_recurrence} to calculate
    \begin{align}
        \delta^\omega(0)[n]&=\delta^\omega(0)[mQ+r]\\
        &=\delta(\delta^\omega(0)[m])[r]\\
        &=\delta(f_q(m))[r]\\
        &=f_q(m)+f_q(r)\\
        &=f_q(Qm)+f_q(r)\\
        &=f_q(Qm+r)\\
        &=f_q(n)
    \end{align}

    The third equality is by our induction hypothesis, and the fourth is by the definition of $\delta$. The fifth is because $\gcd(p,q)=1$, so $Q=p^{\varphi(q)}\equiv_q1$. This completes the induction, so for any $k,n\in\mb{N}$ with $n<Q^k$, $\delta^\omega(0)[n]=f_q(n)$. Thus, $\delta^\omega(0)=f_q$.
    
\end{proof}

Now we also express $t_q$ as the fixed point of a $Q$-uniform morphism.

\begin{proposition}\label{prop:phi_is_uniform}
    Treat $A_p=\mb{Z}/p\mb{Z}$ as an alphabet. Define the morphism $\mu$ on $A_p^*$ by $\mu(a)=\phi^{\varphi(q)}(a)$ for all $a\in A_p$, where $\phi$ is the morphism defined by $\phi(a)=\concat_{j=0}^{p-1}(a+f_p(j))$. Then $\mu$ is a $Q$-uniform morphism with $\mu^\omega(0)=t_p$.
\end{proposition}

\begin{proof}
    The morphism $\phi$ is $p$-uniform, so $|\phi^k(a)|=p^k$ for any $k\in\mb{N}$. This means that $|\mu(a)|=|\phi^{\varphi(q)}(a)|=p^{\varphi(q)}=Q$ for any $a\in A$, so $\mu$ is a $Q$-uniform morphism.

    Now we will show that $\mu^\omega(0)[:Q^k]=\phi^\omega(0)[:Q^k]$ for all $k\in\mb{N}$ by induction on $k$. In the base case where $k=0$, we have $\mu^\omega(0)[:1]=0=\phi^\omega(0)[:1]$. Assume that $\mu^\omega(0)[:Q^k]=\phi^\omega(0)[:Q^k]$ for some arbitrary $k\in\mb{N}$. Then we have
    \begin{align*}
        \mu^\omega(0)[:Q^k]&=\phi^\omega(0)[:Q^k]\\
        \mu(\mu^\omega(0)[:Q^k])&=\mu(\phi^\omega(0)[:Q^k])\\
        \mu^\omega(0)[:Q^{k+1}]&=\phi^{\varphi(q)}(\phi^\omega(0)[:Q^k])\\
        \mu^\omega(0)[:Q^{k+1}]&=\phi^\omega(0)[:Q^k\cdot p^{\varphi(q)}]\\
        \mu^\omega(0)[:Q^{k+1}]&=\phi^\omega(0)[:Q^{k+1}].
    \end{align*}

    This completes the induction, so $\mu^\omega(0)[:Q^k]=\phi^\omega(0)[:Q^k]$ for all $k\in\mb{N}$. It follows that $\mu^\omega(0)[n]=\phi^\omega(0)[n]$ for all $n\in\mb{N}$. Thus, $\mu^\omega(0)=\phi^\omega(0)$. By Proposition \ref{prop:t_p_equivalent_defs}, $\phi^\omega(0)=t_p$, so we also have $\mu^\omega(0)=t_p$.
\end{proof}

By Propositions \ref{prop:delta_is_uniform} and \ref{prop:phi_is_uniform}, we know that we can express both $f_q$ and $t_p$ as the fixed points of $Q$-uniform morphisms. Proposition \ref{prop:combine_morphisms} will put these results together to describe Dekking sequences as a fixed point of a $Q$-uniform morphisms. Proposition \ref{prop:combine_morphisms} here is analogous to Proposition 1 in Dekking's original paper \cite{dekking}.

\begin{proposition}\label{prop:combine_morphisms}
    Treat $A_{p,q}=\mb{Z}/p\mb{Z}\times\mb{Z}/q\mb{Z}$ as an alphabet, and define the morphism $\lambda$ on $A_{p,q}^*$ by $\lambda((x,y))=\concat_{j=0}^{Q-1}(\mu(x)[j],\delta(y)[j])$ for all $(x,y)\in A_{p,q}$. Then  $\lambda^\omega((0,0))=z_{p,q}$, where $z_{p,q}$ is a \textit{Dekking sequence}.
\end{proposition}

\begin{proof}
    We will show that for any $k\in\mb{N}$, $\lambda^\omega((0,0))[n]=z_{p,q}(n)$ for all $0\leq n<Q^k$ by induction on $k$. The base case $k=0$ holds because $\lambda^\omega((0,0))[0]=(0,0)=z_{p,q}(0)$.
    
    Now assume that for some arbitrary $k\in\mb{N}$, $\lambda^\omega((0,0))[n]=z_{p,q}(n)$ for all $0\leq n<Q^k$. Consider some $n\in\mb{N}$ with $n<Q^{k+1}$. Then we can write $n=mQ+r$ for some $m,r\in\mb{N}$ with $m<Q^k$ and $r<Q$. Using Lemma \ref{lem:morphism_is_recurrence}, Propositions \ref{prop:delta_is_uniform} and \ref{prop:phi_is_uniform}, and properties of the fixed points of morphisms, we find
    \begin{align}
        \lambda^\omega((0,0))[n]&=\lambda^\omega((0,0))[mQ+r]\\
        &=\lambda(\lambda^\omega((0,0))[m])[r]\\
        &=\lambda(z_{p,q}(m))[r]\\
        &=\lambda((t_p(m),f_q(m)))[r]\\
        &=\lambda((\mu^\omega(0)[m],\delta^\omega(0)[m]))[r]\\
        &=(\mu(\mu^\omega(0)[m])[r],\delta(\delta^\omega(0)[m])[r])\\
        &=(\mu^\omega(0)[mQ+r],\delta^\omega(0)[mQ+r])\\
        &=(\mu^\omega(0)[n],\delta^\omega(0)[n])\\
        &=(t_p(n),f_q(n))\\
        &=z_{p,q}(n).
    \end{align}
    Equality (2) is by Lemma \ref{lem:morphism_is_recurrence}, (3) is by the induction hypothesis, (5) is by Propositions 3 and 4, (6) is by the definition of $\lambda$ and the fact that $r<Q$, and (7) is again by Lemma \ref{lem:morphism_is_recurrence}. This completes the induction, so we have $\lambda^\omega((0,0))[n]=z_{p,q}(n)$ for all $n\in\mb{N}$. Therefore, $\lambda^\omega((0,0))=z_{p,q}$.
\end{proof}

The existence of the morphism $\lambda$ will help us to better understand Dekking sequences when they are put into a geometric context.

\subsection{Dekking Curves}

Now that we have defined Dekking sequences, we can use them to describe the sums $Z(N,p,q)$ as absolute turtle curves.

\begin{definition}
    Let $p,q\geq2$ and $k$ be integers with $\gcd(k,q)=1$. Let $\zeta_p=e^{\frac{2\pi i}{p}}$ and $\zeta_q=e^{\frac{2\pi i}{q}}$. The \textit{Dekking curve} $D_{p,q,k}$ is defined to be the absolute turtle curve $D_{p,q,k}=(z_{p,q},\tau)$, where $\tau((x,y))=\zeta_p^x\zeta_q^{ky}$ for all $(x,y)\in A_{p,q}$.
\end{definition}

For a visual, Figure \ref{fig:D_curves} shows some examples of Dekking curves. Similar figures appear in Dekking's original paper \cite{dekking}.

\begin{figure}[H]
    \centering
    \includegraphics[scale=0.6]{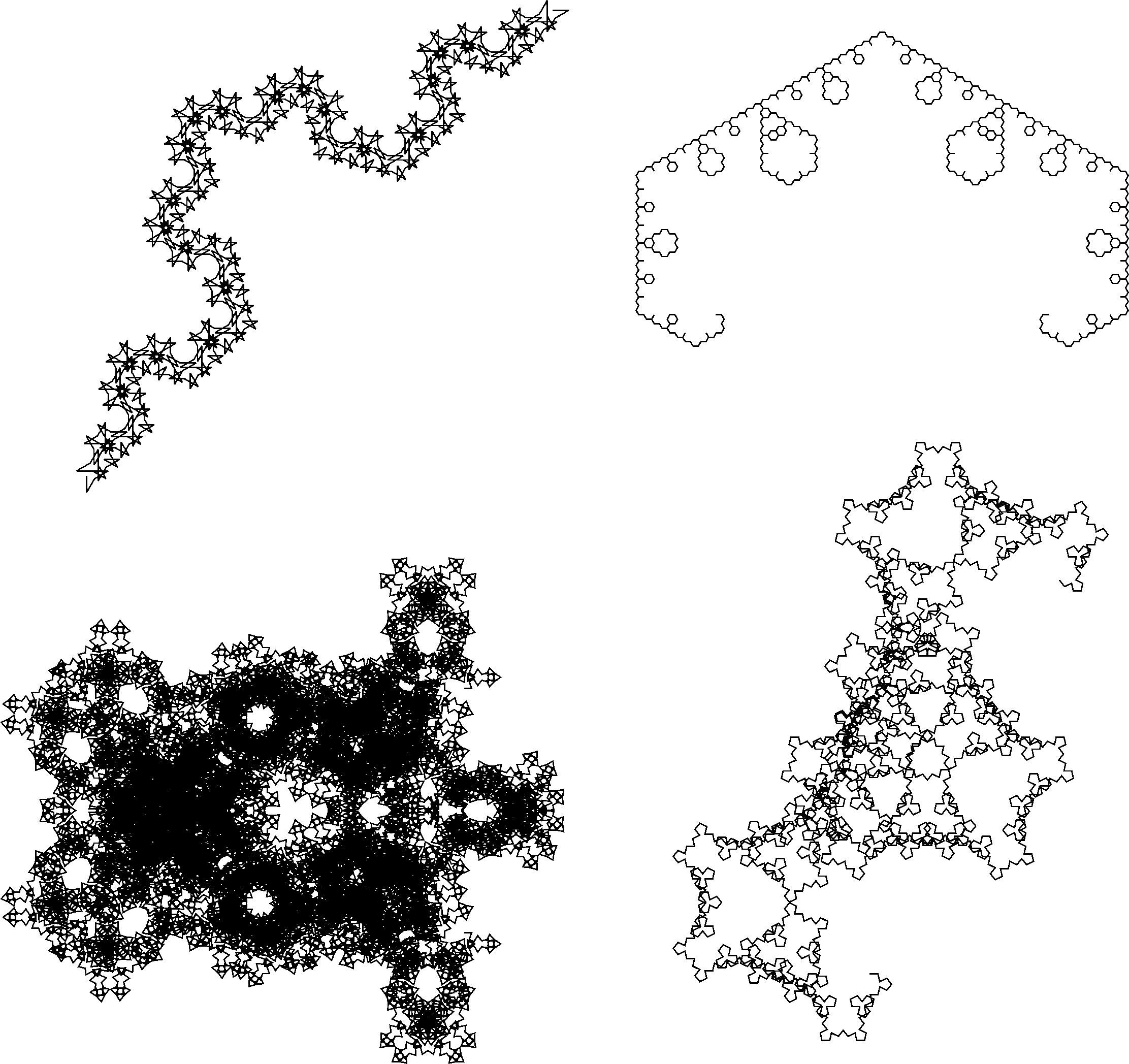}
    \caption{The first $2^{10}$ steps of $D_{2,12,1}$ (top left), the first $3^{6}$ steps of $D_{3,2,1}$ (top right), the first $2^{15}$ steps of $D_{2,31,6}$ (bottom left), and the first $2^{11}$ steps of $D_{2,7,2}$ (bottom right). The curve $D_{2,31,6}$ was chosen because it resembles a sea turtle.}
    \label{fig:D_curves}
\end{figure}

In more explicit terms, the $N$th point of $D_{p,q,k}$ can be evaluated as
\begin{align*}
    D_{p,q,k}(N)&=\sum_{i=0}^{N-1}\tau(z_{p,q}(i))\\
    &=\sum_{i=0}^{N-1}\tau((t_p(i),f_q(i)))\\
    &=\sum_{i=0}^{N-1}\zeta_p^{t_p(i)}\zeta_q^{kj}.
\end{align*}

Observe that in the case where $k=1$, $D_{p,q,1}(N)=Z(N,p,q)$, so our definition coincides with the sums studied by Dekking in \cite{dekking}. One nice property of Dekking curves is that each term in the sum is a root of unity. Even better, we have the following fact.

\begin{proposition}\label{prop:P_D_morphism_to_S1}
    Let $D=D_{p,q,k}$ be a Dekking curve. Then the function $P_D$ is a homomorphism from $(A_{p,q},+)$ to $(\mb{S}^1,\cdot)$.
\end{proposition}

\begin{proof}
    For any $(x_1,y_1),(x_2,y_2)\in A_{p,q}$, we have
    \begin{align*}
        P_D((x_1,y_1)+(x_2,y_2))&=P_D((x_1+x_2,y_1+y_2))\\
        &=\zeta_p^{x_1+x_2}\zeta_q^{k(y_1+y_2)}\\
        &=(\zeta_p^{x_1}\zeta_q^{ky_1})(\zeta_p^{x_2}\zeta_q^{ky_2})\\
        &=P_D((x_1,y_1))P_D((x_2,y_2)).
    \end{align*}
\end{proof}

As one might notice in Figure \ref{fig:D_curves}, Dekking curves exhibit a fractal-like structure, with smaller versions of the entire curve visible within each curve. This is ultimately a consequence of the fact that the sequence $z_{p,q}$ is morphic, as revealed by Proposition \ref{prop:D_self_similar}.

\begin{proposition}\label{prop:D_self_similar}
    Let $D=D_{p,q,k}$ be a Dekking curve with $\gcd(p,q)=1$. Let $r=P_D(\lambda((0,0)))$, where $\lambda$ is the $Q$-uniform morphism defined in Proposition \ref{prop:combine_morphisms}. Then $D(Qn)=rD(n)$ for all $n\in\mb{N}$.
\end{proposition}

\begin{proof}

Using the fact that $z_{p,q}$ is the fixed point of the morphism $\lambda$ and $P_D$ is a homomorphism, we have
\begin{align*}
    D(Qn)&=P_D(z_{p,q}[:Qn])\\
    &=P_D(\lambda(z_{p,q}[:n]))\\
    &=P_D(\lambda(\concat_{i=0}^{n-1}z_{p,q}(i)))\\
    &=P_D(\concat_{i=0}^{n-1}\lambda(z_{p,q}(i)))\\
    &=\sum_{i=0}^{n-1}P_D(\lambda(z_{p,q}(i)))
\end{align*}

Recall that the morphism $\lambda$ is defined by $\lambda((x,y))=\concat_{j=0}^{Q-1}(\mu(x)[j],\delta(y)[j])$, where $\mu$ and $\delta$ are the morphisms defined in Propositions \ref{prop:phi_is_uniform} and \ref{prop:delta_is_uniform} respectively. From those definitions, we have
\begin{align*}
    \lambda((x,y))&=\concat_{j=0}^{Q-1}(x+t_p(j),y+f_q(j))\\
    &=\concat_{j=0}^{Q-1}\left((x+y)+(t_p(j),f_q(j))\right)\\
    &=\concat_{j=0}^{Q-1}\left((x+y)+\lambda((0,0))[j]\right)
\end{align*}

We can use this to continue our previous calculation.
\begin{align*}
    D(Qn)&=\sum_{i=0}^{n-1}P_D(\lambda(z_{p,q}(i)))\\
    &=\sum_{i=0}^{n-1}P_D(\concat_{j=0}^{Q-1}(z_{p,q}(i)+\lambda((0,0))[j])))\\
    &=\sum_{i=0}^{n-1}\sum_{j=0}^{Q-1}P_D(z_{p,q}(i)+\lambda((0,0))[j]))
\end{align*}

By Proposition \ref{prop:P_D_morphism_to_S1}, we can view $P_D$ has a homomorphism from $(A_{p,q},+)$ to $(\mb{S}^1,\cdot)$, which yields our desired result.
\begin{align*}
    D(Qn)&=\sum_{i=0}^{n-1}\sum_{j=0}^{Q-1}P_D(z_{p,q}(i))P_D(\lambda((0,0))[j])\\
    &=\sum_{i=0}^{n-1}P_D(z_{p,q}(i))\sum_{j=0}^{Q-1}P_D(\lambda((0,0))[j])\\
    &=\sum_{i=0}^{n-1}P_D(z_{p,q}(i))P_D\left(\concat_{j=0}^{Q-1}\lambda((0,0))[j]\right)\\
    &=\sum_{i=0}^{n-1}P_D(z_{p,q}(i))P_D(\lambda((0,0)))\\
    &=\sum_{i=0}^{n-1}P_D(z_{p,q}(i))r\\
    &=r\sum_{i=0}^{n-1}P_D(z_{p,q}(i))\\
    &=rD(n)
\end{align*}

\end{proof}

Note that because of the self similar property shown in Proposition \ref{prop:D_self_similar}, the value of $r=P_D(\lambda((0,0)))$ will play an important role in determining the large scale geometric structure of Dekking curves. In order for scaled segments of a Dekking curve to converge to some limit curve, the Dekking curve will need to have the following property.

\begin{definition}
    Let $D=D_{p,q,k}$ be a Dekking curve with $\gcd(p,q)=1$. Then the \textit{scaling factor} of $D$ is $r=P_D(\lambda((0,0)))$, where $\lambda$ is the $Q$-uniform morphism defined in Proposition \ref{prop:combine_morphisms}. The Dekking curve $D$ is said to be \textit{regular} if $|r|>1$.
\end{definition}

The convergence of regular Dekking curves will be proven in Theorem \ref{thm:dekking_convergence}.

\section{Results}

In this section we introduce the concept of two turtle curves being similar, and we show that scaled segments of similar turtle curves will converge to the same fractal curve in the Hausdorff metric. We then establish several results to show the similarity of certain Thue-Morse curves and Dekking curves, and conclude that those curves will converge to the same limit curves.

\subsection{Similarity of Turtle Curves}

Experimentally, it is often easy to see when two turtle curves share the same underlying geometric structure. For example, the two curves in Figure \ref{fig:similar_TCs} clearly have a resemblance, even if they look very different on a small scale. To make this idea formal, we introduce the following definition.

\begin{definition}
    Two turtle curves $T_1,T_2$ are called \textit{similar} if there exist positive integers $k_1,k_2$ and $c\in\mb{C}^\times$ such that $c\cdot T_1(k_1n)=T_2(k_2n)$ for all $n\in\mb{N}$. We then write $T_1\sim T_2$.
\end{definition}

So two turtle curves are similar if there is a way of scaling them such that their points will periodically coincide. 

\begin{figure}[H]
    \centering
    \includegraphics[scale=0.6]{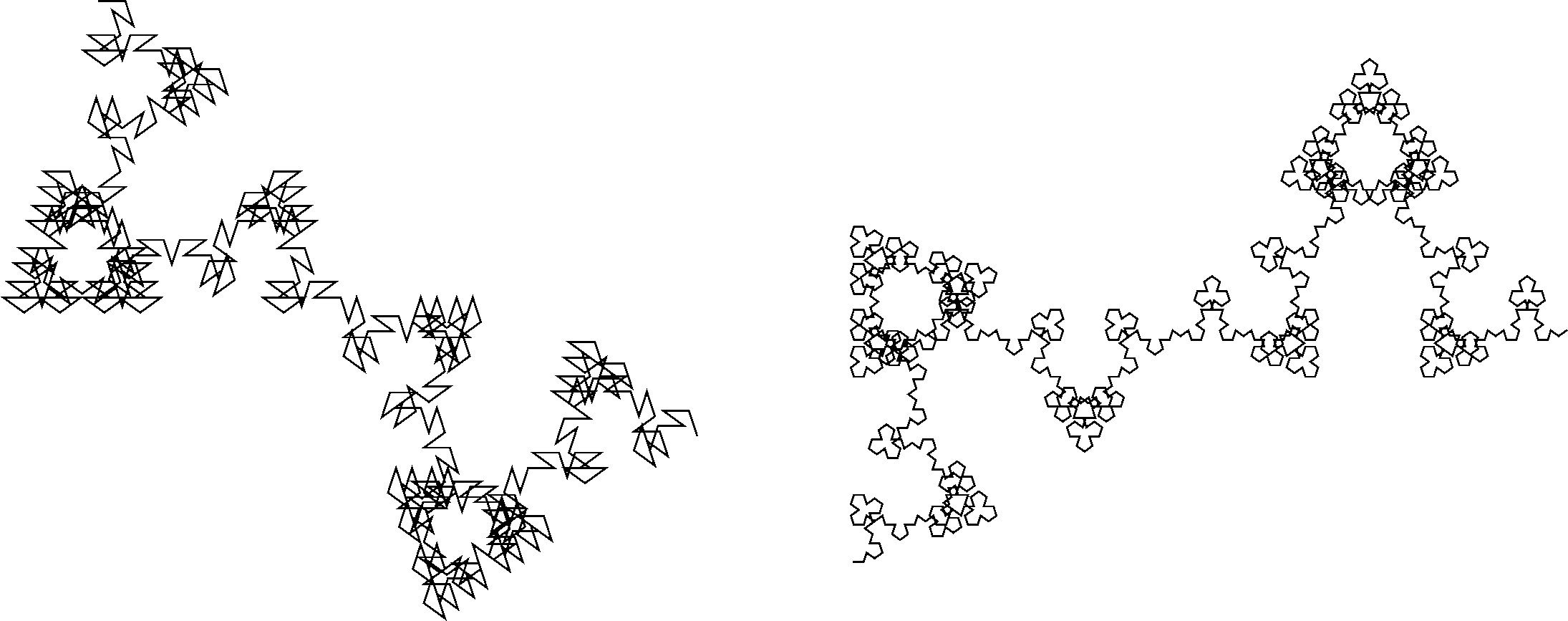}
    \caption{On the left is the Thue-Morse turtle curve $T=(t_2,\tau)$, where $\tau(0)=(1,\zeta_5^2)$ and $\tau(1)=(0,\zeta_5)$. On the right is the Dekking curve $D_{2,10,7}$. Theorem \ref{thm:main_result} will show that $T\sim D_{2,10,7}$.}
    \label{fig:similar_TCs}
\end{figure}

We also have the following proposition, which will be very helpful when we show similarity between Thue-Morse Curves and Dekking curves.

\begin{proposition}
    The similarity relation $\sim$ is an equivalence relation.
\end{proposition}

\begin{proof}
    We can see that $\sim$ is reflexive by taking $c=1$ and $k_1=k_2=1$. It is also easy to see that $\sim$ is reflexive, as $c$ is required to be invertible. To show transitivity, let $T_1,T_2,T_3$ be turtle curves with $T_1\sim T_2$ and $T_2\sim T_3$. Then by definition, there exist $c,d\in\mb{C}$ and positive integers $k_1,k_2,m_1,m_2$ such that
    \begin{equation}
        c\cdot T_1(k_1n)=(T_2(k_2n)\tag{1}
    \end{equation}
    and
    \begin{equation}
        d\cdot T_2(m_1n)=T_3(m_2n)\tag{2}
    \end{equation}
    for all $n\in\mb{N}$. Because (1) is true for all $n\in\mb{N}$, we can replace $n$ with $m_1n$, yielding that
    \begin{equation}
        c\cdot T_1(k_1m_1n)=T_2(k_2m_1n)\tag{3}
    \end{equation}
    for all $n\in\mb{N}$. Similarly, because (2) is true for all $n\in\mb{N}$, we can replace $n$ with $k_2n$, yielding that
    \begin{equation}
        d\cdot T_2(k_2m_1n)=T_3(k_2m_2n)\tag{4}
    \end{equation}
    for all $n\in\mb{N}$. Then substituting (3) into (4) gives that
    \begin{equation*}
        cd\cdot T_1(k_1m_1n)=T_3(k_2m_2n)
    \end{equation*}
    for all $n\in\mb{N}$. Because $c,d\neq0$, $cd\neq0$, and because $k_1,k_2,m_1,m_2$ are all positive integers, $k_1m_1$ and $k_2m_2$ are both positive integers. Therefore, $T_1\sim T_3$, and we have shown that $\sim$ is an equivalence relation.
\end{proof}

\subsection{Convergence of Turtle Curves}

In this section we will establish several results on the convergence of scaled turtle curves in the Hausdorff metric. In particular, establish the convergence of scaled Dekking curves, and we show that similar turtle curves will share the same limit curves. To do this, we will need a concept of distance between subsets of $\mb{C}$, which can be achieved using the Hausdorff metric. For our purposes, the Hausdorff metric on $\mb{C}$ is defined as follows.

\begin{definition}
    For any $x,y\in\mb{C}$ define $d(x,y)=|x-y|$. Let $\mc{H}(\mb{C})$ be the set of all nonempty subsets of $\mb{C}$ which are compact with respect to $d$. For any $x\in\mb{C}$ and $Y\in\mc{H}(\mb{C})$, define $d(x,Y)=\inf\{d(x,y):y\in Y\}$. Then the \textit{Hausdorff distance} between any $X,Y\in\mc{H}(\mb{C})$ is defined as
    $$d_H(X,Y)=\max\{\sup\{d(x,Y):x\in X\},\sup\{d(y,X):y\in Y\}\}.$$
\end{definition}

From this definition, the following fact is well known.

\begin{proposition}
    The set $\mc{H}(\mb{C})$ equipped with the distance metric $d_H$ forms a complete metric space.
\end{proposition}

Note that for any turtle curve $T$ and $n\in\mb{N}$, the set $\mc{P}(T[:n])$ is compact, as it is a finite union of compact sets $\ell(T(i),T(i+1))$. We also know $\mc{P}(T[:n])$ is nonempty because we will always have $0\in\ell(T(0),T(1))\subseteq\mc{P}(T[:n])$. Therefore, the Hausdorff metric can be used to measure the distance between any two finite turtle curve segments. This allows for the following definition.

\begin{definition}
    We say that $K\in\mc{H}(\mb{C})$ is a \textit{limit curve} of a turtle curve $T$ if there exists $r,c\in\mb{C}^\times$ with $|r|>1$ and a sequence of natural numbers $(b_n)$ such that
    $$\lim_{n\to\infty}r^{-n}\mc{P}(T[:b_n])=cK$$
    in the Hausdorff metric. We say that $K$ is \textit{nontrivial} if $K\neq\{0\}$.
\end{definition}

Note that a turtle curve $T$ has $\{0\}$ as a limit curve if and only if $\mc{P}(T)$ is bounded. The following theorem establishes that certain Dekking curves have a nontrivial limit curve.

\begin{theorem}\label{thm:dekking_convergence}
    Let $D=D_{p,q,k}$ be a regular Dekking curve with scaling factor $r$, and let $Q=p^{\varphi(q)}$. Then $D$ has a nontrivial limit curve $K$ with
    $$\lim_{n\to\infty}r^{-n}\mc{P}(D[:Q^n])=K.$$
\end{theorem}

\begin{proof}
    For any $n\in\mb{N}$, let $S_n=r^{-n}\mc{P}(D[:Q^n])$. We would like to show that there exists some $K\neq\{0\}$ such that
    $$\lim_{n\to\infty}S_n=K$$
    in the Hausdorff metric. First we will show that this limit exists. Because the Hausdorff metric is complete, it suffices to show that the sequence $(S_n)$ is Cauchy. We will start by finding an upper bound for $d_H(S_n,S_{n+1})$ in terms of $n$. 

    Let $x\in S_n$. We would like to find an upper bound on $d(x,S_{n+1})$. First see that
    $$x\in S_n=r^{-n}\mc{P}(D[:Q^n])=r^{-n}\bigcup_{i=0}^{Q^n-1}\ell(D(i),D(i+1))=\bigcup_{i=0}^{Q^n-1}r^{-n}\ell(D(i),D(i+1)),$$
    so there exists some integer $0\leq m<Q^n$ such that $x\in r^{-n}\ell(D(m),D(m+1))$. Multiplying by $r^n$, we have $r^nx\in\ell(D(m),D(m+1))$. The distance between consecutive points of a Dekking curve is always 1, so  $r^nx\in\ell(D(m),D(m+1))$ implies that $d(r^nx,D(m))\leq1$. Multiplying by $r^{-n}$, we have $d(x,r^{-n}D(m))\leq r^{-n}$. By Proposition \ref{prop:D_self_similar}, we know $r^{-n}D(m)=r^{-(n+1)}D(Qm)$. Since $0\leq m<Q^n$, we have $0\leq Qm<Q^{n+1}$. But then
    $$r^{-(n+1)}D(Qm)\in r^{-(n+1)}\ell(D(Qm), D(Qm+1))\subseteq S_{n+1}.$$
    So because $d(x,r^{-(n+1)}D(Qm))=d(x,r^{-n}D(m))\leq r^{-n}$ and $r^{-(n+1)}D(Qm)\in S_{n+1}$, we have found that
    $$d(x,S_{n+1})=\inf\{d(x,s):s\in S_{n+1}\}\leq d(x,r^{-(n+1)}D(Qm))\leq r^{-n}.$$
    The choice of $x\in S_n$ was arbitrary, so we have
    $$\sup\{d(x,S_{n+1}):x\in S_n\}\leq r^{-n}.$$

    Now let $x\in S_{n+1}$. We would like to find a bound on $d(x,S_n)$. Similar to before, we have
    $$x\in S_{n+1}=r^{-(n+1)}\mc{P}(D[:Q^{n+1}])=\bigcup_{i=0}^{Q^{n+1}-1}r^{-(n+1)}\ell(D(i),D(i+1)),$$
    so there exists some integer $0\leq m<Q^{n+1}$ such that $x\in r^{-(n+1)}\ell(D(m),D(m+1))$. Because $0\leq m<Q^{n+1}$, we can use the division algorithm to find integers $0\leq m_1<Q^n$ and $0\leq m_2<Q$ such that $m=Qm_1+m_2$. Since $x\in r^{-(n+1)}\ell(D(m),D(m+1))$, we have
    $$r^{n+1}x\in\ell(D(m),D(m+1))=\ell(D(Qm_1+m_2),D(Qm_1+m_2+1)).$$
    Then because the distance between consecutive points of $D$ is 1,
    $$d(r^{n+1}x,D(Qm_1+m_2))\leq 1.$$
    Because $m_2$ is an integer less than $Q$, by the triangle inequality, we have
    $$d(D(Qm_1+m_2),D(Qm_1))\leq(Q-1).$$
    Adding this to the previous inequality and using the triangle inequality again gives
    \begin{align*}
        d(r^{n+1}x,D(Qm_1))&\leq d(r^{n+1}x,D(Qm_1+m_2))+d(D(Qm_1+m_2),D(Qm_1))\\
        &\leq 1+(Q-1)\\
        &=Q.
    \end{align*}
    Taking this result and multiplying by $r^{-(n+1)}$ tells us that
    $$d(x,r^{-(n+1)}D(Qm_1))\leq r^{-(n+1)}Q.$$
    Applying Proposition \ref{prop:D_self_similar} to the left side then gives
    $$d(x,r^{-n}D(m_1))\leq r^{-(n+1)}Q.$$
    Since $0\leq m_1<Q^n$, we have 
    $$r^{-n}D(m_1)\in\ell(D(m_1),D(m_1+1))\subseteq S_n.$$
    So because $d(x,r^{-n}D(m_1))\leq r^{-(n+1)}Q$ and $r^{-n}D(m_1)\in S_n$, we have found that
    $$D(x,S_n)\leq r^{-(n+1)}Q.$$
    The choice of $x\in S_{n+1}$ was arbitrary, so we have
    $$\sup\{d(x,S_n):x\in S_{n+1}\}\leq r^{-(n+1)}Q.$$
    Recall that our previous argument established that
    $$\sup\{d(x,S_{n+1}):x\in S_n\}\leq r^{-n}.$$
    Because $|r|>1$ and $Q\geq1$, we have both $r^{-(n+1)}Q\leq r^{-n}Q$ and $r^{-n}\leq r^{-n}Q$, so the value $r^{-n}Q$ is an upper bound on both supremums. That is,
    $$d_H(S_n,S_{n+1})=\max\{\sup\{d(x,S_n):x\in S_{n+1}\},\sup\{d(x,S_{n+1}):x\in S_n\}\}\leq r^{-n}Q.$$
    Now that we have bounded the distance between consecutive terms of the sequence $(S_n)$ in terms of $n$, we will use this result to argue that $(S_n)$ is Cauchy. Let $n,m\in\mb{N}$ with $n<m$. Then using the triangle inequality and the upper bound we found on $d_H(S_n,S_{n+1})$, we have
    \begin{align*}
        d_H(S_n,S_m)&\leq\sum_{i=n}^{m-n-1}d_H(S_i,S_{i+1})\\
        &\leq\sum_{i=n}^{m-1}r^{-i}Q\\
        &=r^{-n}Q\sum_{i=0}^{m-n-1}r^{-i}\\
        &<r^{-n}Q\sum_{i=0}^\infty r^{-i}\\
        &=r^{-n}\frac{Q}{1-\frac{1}{r}}.
    \end{align*}
    Note that the infinite series above converges because $D$ is regular, meaning $|r|>1$. Thus, we have shown that for any $n,m\in\mb{N}$ with $n<m$,
    $$d_H(S_n,S_m)<r^{-n}\frac{Q}{1-\frac{1}{r}}.$$
    Because $|r|>1$, for any $\epsilon>0$, we can choose an $N\in\mb{N}$ such that $r^{-N}\frac{Q}{1-\frac{1}{r}}<\epsilon$. Then for any $m>n>N$, we will have
    $$d_H(S_n,S_m)<r^{-n}\frac{Q}{1-\frac{1}{r}}<r^{-N}\frac{Q}{1-\frac{1}{r}}<\epsilon.$$
    Therefore, $(S_n)$ is Cauchy. The Hausdorff metric is complete, so this implies that the sequence $(S_n)$ converges to some $K\in\mc{H}(\mb{C})$. To see that $K\neq\{0\}$, simply note that by Proposition \ref{prop:D_self_similar}, for any $n\in\mb{N}$,
    $$D(1)=r^{-n}D(Q^n)\in r^{-n}\ell(D(Q^n-1),D(Q^n))\subseteq S_n.$$
    We always have $D(1)=P_D((0,0))=1$, so $1\in S_n$ for all $n\in\mb{N}$. This implies that $d_H(S_n,\{0\})\geq1$, so the sequence $(S_n)$ cannot converge to $\{0\}$.
\end{proof}

In particular, the Dekking curve $D=D_{2,3,1}$ is regular, and it is in fact the classical approximation of the Koch curve using line segments. We can compute that $Q=2^{\varphi(3)}=4$ and that $r=P_D(\lambda((0,0)))=D(4)=3$. Thus, nontrivial limit curve $K$ of $D_{2,3,1}$ guaranteed by Theorem \ref{thm:dekking_convergence} is
$$\lim_{n\to\infty}3^{-n}\mc{P}(D[:4^n])=K,$$
which is the Koch fractal curve. Figure \ref{fig:dekking_convergence} shows several of the sets $3^{-n}\mc{P}(D[:4^n])$.

\begin{figure}[H]
    \centering
    \includegraphics[width=0.7\linewidth]{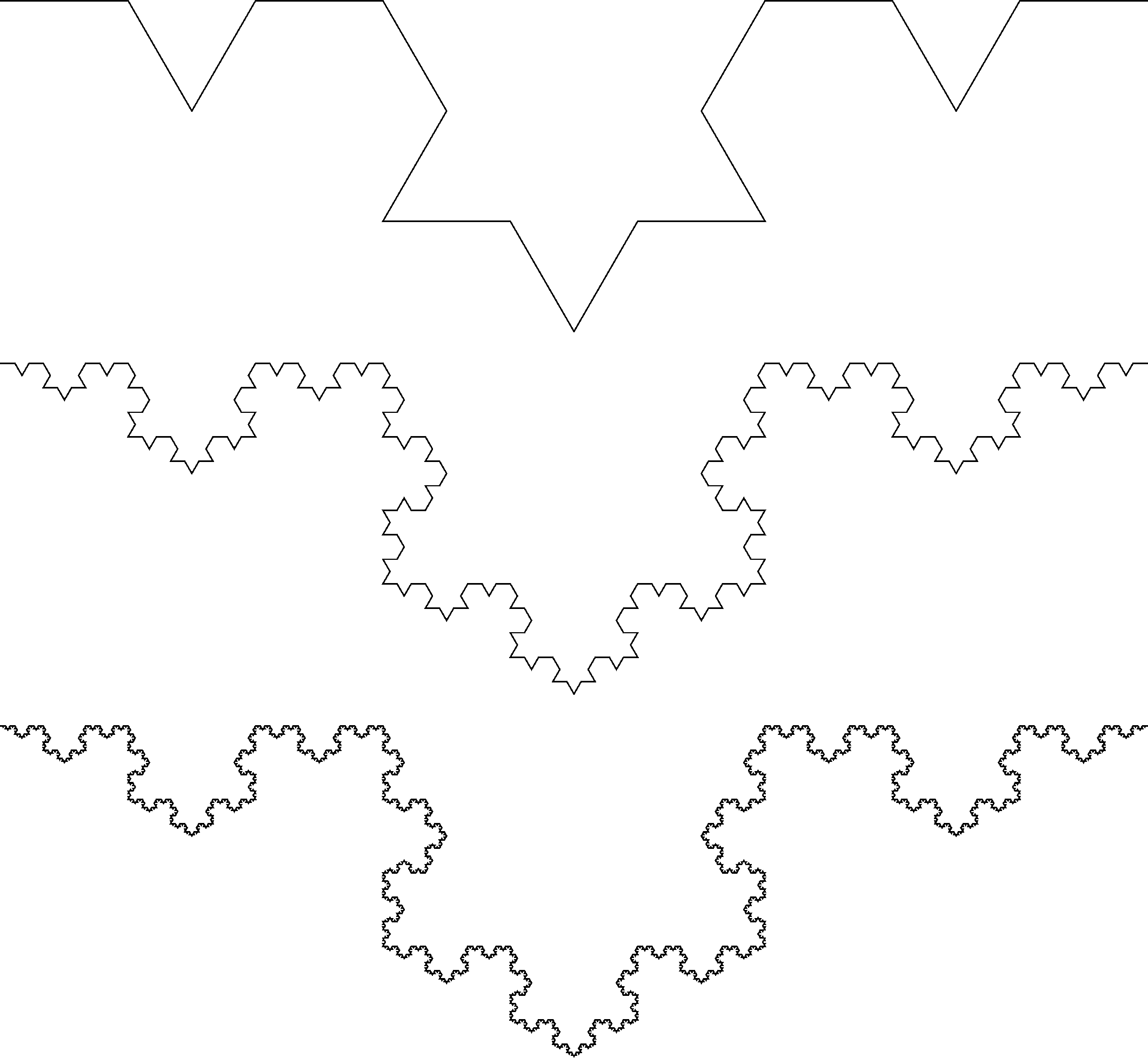}
    \caption{The sets $3^{-n}\mc{P}(D_{3,2,1}[:4^n])$ for $n=2,4,6$.}
    \label{fig:dekking_convergence}
\end{figure}

The next theorem is our main motivation for defining the similarity of turtle curves. It states that similar turtle curves have exactly the same limit curves.

\begin{theorem}\label{thm:sim_convergence}
    Let $T_1$ and $T_2$ be turtle curves with $T_1\sim T_2$. Then for any set $K\in\mc{H}(\mb{C})$, $K$ is a limit curve of $T_1$ if and only if $K$ is a limit curve of $T_2$.
\end{theorem}

\begin{proof}
    Because $\sim$ is a symmetric relation, it suffices to show that if $K\in\mc{H}(\mb{C})$ is a limit curve of $T_1$, it must be a limit curve of $T_2$. So assume that $K$ is a limit curve of $T_1$. Then there exist $r,c_1\in\mb{C}^\times$ with $|r|>1$ and a sequence of natural numbers $(b_n)$ such that
    $$\lim_{n\to\infty}r^{-n}\mc{P}(T_1[:b_n])=K.$$
    As $T_1\sim T_2$, there exist positive integers $k_1,k_2$ and $c_2\in\mb{C}^\times$ such that $$c_2\cdot T_1(k_1n)=T_2(k_2n)$$
    for all $n\in\mb{N}$. Using the division algorithm, for each $n\in\mb{N}$, write $b_n=k_1\tilde{b}_n+j_n$, where $\tilde{b}_n$ and $j_n$ are integers with $0\leq j_n<k_1$. Define the sequence $(a_n)$ by letting $a_n=k_2\tilde{b}_n$ for any $n\in\mb{N}$.

    The turtle curve $T_1$ is defined using some alphabet $A_1$, which is finite, so we can define $L_1=\max\{|P_{T_1}(a)|:a\in A_1\}$. Similarly, let $L_2=\max\{|P_{T_2}(a)|:a\in A_2\}$, where $A_2$ is the alphabet used by the turtle curve $T_2$. The constants $L_1$ and $L_2$ give a maximum distance between consecutive points of $T_1$ and consecutive points of $T_2$ respectively. We claim that for any $n\in\mb{N}$, there is an upper bound
    $$d_H(c_2\mc{P}(T_1[:b_n]),\mc{P}(T_2[:a_n]))\leq\max\{c_2k_1L_1,k_2L_2\}.$$
    To see why, consider an arbitrary $x\in c_2\mc{P}(T_1[:b_n])$. Then
    $$x\in c_2\mc{P}(T_1[:b_n])=\bigcup_{i=0}^{b_n-1}c_2\ell(T_1(i),T_1(i+1)),$$
    so there exists some $0\leq m<b_n$ such that $x\in c_2\ell(T_1(m),T_1(m+1))$. Dividing by $c_2$, we have $c_2^{-1}x\in\ell(T_1(m),T_1(m+1))$. Then because $L_1$ is an upper bound on the distance between consecutive points of $T_1$, we have $d(c_2^{-1}x,T_1(m))\leq L_1$. Using the division algorithm, write $m=k_1m_1+m_2$ for some integers $m_1$ and $0\leq m_2<k_1$. Because $m_2<k_1$, using the triangle inequality, we have 
    $$d(T_1(k_1m_1+m_2),T_1(k_1m_1))\leq m_2L_1\leq(k_1-1)L_1.$$
    Using our previous finding that $d(c_2^{-1}x,T_1(m))\leq L_1$ along with the triangle inequality then gives
    \begin{align*}
        d(c_2^{-1}x,T_1(k_1m_1))&\leq d(c_2^{-1}x,T_1(m))+d(T_1(k_1m_1+m_2),T_1(k_1m_1))\\
        &\leq L_1+(k_1-1)L_1\\
        &=k_1L_1.
    \end{align*}
    Taking this result and multiplying by $c_2$ then gives
    $$d(x,c_2T_1(k_1m_1))\leq c_2k_1L_1.$$
    Because $T_1\sim T_2$, we have $c_2T_1(k_1m_1)=T_2(k_2m_1)$. Recall that we used the division algorithm on both $m$ and $b_n$ with divisor $k_1$, giving $m=k_1m_1+m_2$ and $b_n=k_1\tilde{b}_n+j_n$. Because $m<b_n$, we must have $m_1\leq\tilde{b}_n$. Multiplying by $k_2$, we have $k_2m_1\leq k_2\tilde{b}_n=a_n$. Thus,  $T_2(k_2m_1)\in\mc{P}(T_2[:a_n])$. We had 
    $$d(x,T_2(k_2m_1))=d(x,c_2T_1(k_1m_1))\leq c_2k_1L_1,$$
    so it follows that
    $$d(x,\mc{P}(T_2[:a_n]))\leq c_2k_1L_1.$$
    The choice of $x\in c_2\mc{P}(T_1[:b_n])$ was arbitrary, so
    $$\sup\{d(x,\mc{P}(T_2[:a_n])):x\in c_2\mc{P}(T_1[:b_n])\}\leq c_2k_1L_1.$$
    We can make a similar argument in the other direction. Let $x\in\mc{P}(T_2[:a_n])$. Then
    $$x\in \mc{P}(T_2[:a_n])=\bigcup_{i=0}^{a_n-1}\ell(T_2(i),T_2(i+1)),$$
    so there exists $0\leq m<a_n$ such that $x\in\ell(T_2(m),T_2(m+1))$. Similar to before, $L_2$ is the maximum distance between points of $T_2$, so $d(x,T_2(m))\leq L_2$. We can use the division algorithm to write $m=k_2m_1+m_2$ for some integers $m_1,m_2$ with $0\leq m_2<k_2$. Because $m_2<k_2$, we have $d(T_2(m),T_2(k_2m_1))\leq (k_2-1)L_2$. Then using the triangle inequality, we have
    \begin{align*}
        d(x,T_2(k_2m_1))&\leq d(x,T_2(m))+d(T_2(m),T_1(k_2m_1))\\
        &\leq L_2+(k_2-1)L_2\\
        &=k_2L_2.
    \end{align*}
    Because $T_1\sim T_2$, $T_2(k_2m_1)=c_2T_1(k_1m_1)$. Because $m<a_n$, we have
    $$k_1k_2m_1\leq k_1m<k_1a_n=k_1k_2\tilde{b}_n\leq k_2b_n.$$
    Taking out the factor of $k_2$, we see that $k_1m<b_n$. Thus, 
    $$T_2(k_2m_1)=c_2T_1(k_1m_1)\in\mc{P}(T_1[:b_n]).$$
    This together with our previous finding that $d(x,T_2(k_2m_1))\leq k_2L_2$ tells us that
    $$\sup\{d(x,\mc{P}(T_1[:b_n])):x\in \mc{P}(T_2[:a_n])\}\leq k_2L_2.$$
    In the other direction we found that
    $$\sup\{d(x,\mc{P}(T_2[:a_n])):x\in c_2\mc{P}(T_1[:b_n])\}\leq c_2k_1L_1,$$
    so we have shown
    $$d_H(c_2\mc{P}(T_1[:b_n]),\mc{P}(T_2[:a_n]))\leq\max\{c_2k_1L_1,k_2L_2\}.$$
    Then for any $n\in\mb{N}$, we have
    $$d_H(r^{-n}c_2\mc{P}(T_1[:b_n]),r^{-n}\mc{P}(T_2[:a_n]))\leq |r|^{-n}\max\{c_2k_1L_1,k_2L_2\}.$$
    Because $|r|>1$, taking the limit as $n\to\infty$,
    $$\lim_{n\to\infty}d_H(r^{-n}c_2\mc{P}(T_1[:b_n]),r^{-n}\mc{P}(T_2[:a_n]))=0.$$
    Therefore, in the Hausdorff metric,
    $$\lim_{n\to\infty}r^{-n}\mc{P}(T_2[:a_n])=\lim_{n\to\infty}r^{-n}c_2\mc{P}(T_1[:b_n])=c_2K.$$
    Then by definition, $K$ is a limit curve of $T_2$.

\end{proof}

\subsection{Similarity Between Thue-Morse Curves and Dekking Curves}

As a consequence of Theorems \ref{thm:dekking_convergence} and \ref{thm:sim_convergence}, any turtle curve which is similar to a regular Dekking curve $D$ will share a nontrivial limit curve with $D$. This means that if we can show a Thue-Morse turtle curve is similar to a regular Dekking curve, we will have shown that it converges to the same limit curve. Theorem \ref{thm:T_to_B} is the first step to proving this. It shows that any Thue-Morse turtle curve is similar to an absolute turtle curve which uses a Dekking sequence as the source of its instructions.

\begin{theorem}\label{thm:T_to_B}
    Let $T=(t_p,\tau)$ be a Thue-Morse turtle curve such that $\alpha_T(\phi(0))=\zeta_q^k$, where $\phi$ is the morphism used to define $t_p$, and $k$ and $q\geq2$ are coprime. Let $B=(z_{p,q},\kappa)$ be an absolute turtle curve, where $\kappa$ is an interpreter function with $\kappa((x,y))=P_T(\phi(x))\zeta_q^{ky}$ for all $(x,y)\in A_{p,q}$. Then $T\sim B$.
\end{theorem}

\begin{proof}
    We claim that $T(pn)=B(n)$ for all $n\in\mb{N}$. We start by expanding $B(n)$.
    \begin{align*}
        B(n)&=\sum_{i=0}^{n-1}\kappa(z_{p,q}(i))\\
        &=\sum_{i=0}^{n-1}\kappa((t_p(i),f_q(i)))\\
        &=\sum_{i=0}^{n-1}P_T(\phi(t_p(i)))\zeta_q^{ki}
    \end{align*}
    Now consider $T(pn)$. Using the fact that $t_p$ is a fixed point of the $p$-uniform morphism $\phi$, we have
    \begin{align*}
        T(pn)&=P_T(t_p[:pn])\\
        &=P_T(\phi(t_p[:n]))\\
        &=P_T\left(\phi\left(\concat_{i=0}^{n-1}t_p(i)\right)\right)\\
        &=P_T\left(\concat_{i=0}^{n-1}\phi(t_p(i))\right)
    \end{align*}
    This expression can then be evaluated using Lemma \ref{lem:expand_S}, as $P_T$ is simply $\pi_1\circ S_T$.
    $$T(pn)=P_T\left(\concat_{i=0}^{n-1}\phi(t_p(i))\right)=\sum_{i=0}^{n-1}\alpha_T\left(\concat_{j=0}^{i-1}\phi(t_p(j))\right)P_T(\phi(t_p(i)))$$
    Recall that by Proposition \ref{prop:alpha_is_morphism}, $\alpha_T$ is a homomorphism from $(A_p,||)$ to $(\mb{S}^1,\cdot)$, so we have
    $$T(pn)=\sum_{i=0}^{n-1}\alpha_T\left(\concat_{j=0}^{i-1}\phi(t_p(j))\right)P_T(\phi(t_p(i)))=\sum_{i=0}^{n-1}\left[\prod_{j=0}^{i-1}\alpha_T(\phi(t_p(j)))\right]P_T(\phi(t_p(i)))$$
    Our initial assumption was that $\alpha_T(\phi(0))=\zeta_q^k$. But this actually implies that $\alpha_T(\phi(a))=\zeta_q^k$ for any $a\in A_p$, as $\phi(a)$ is just a reordering of the symbols of $\phi(0)$. Thus,
    \begin{align*}
        T(pn)&=\sum_{i=0}^{n-1}\left[\prod_{j=0}^{i-1}\alpha_T(\phi(t_p(j)))\right]P_T(\phi(t_p(i)))\\
        &=\sum_{i=0}^{n-1}\left[\prod_{j=0}^{i-1}\zeta_q^k\right]P_T(\phi(t_p(i)))\\
        &=\sum_{i=0}^{n-1}\zeta_q^{ki}P_T(\phi(t_p(i)))
    \end{align*}
    Observe that this is the exact same sum we obtained by evaluating $B(n)$, so we have $T(pn)=B(n)$ for all $n\in\mb{N}$. Then by definition, $T\sim B$.
\end{proof}

The next theorem provides a way to relate the turtle curve $B$ from Theorem \ref{thm:T_to_B} to a Dekking curve. Because the proof allows it, we will work with a more general $B$ than the $B$ from Theorem \ref{thm:T_to_B}, taking arbitrary complex numbers $c_0,c_1\in\mathbb{C}$ rather than specifically $P_T(\phi(0))$ and $P_T(\phi(1))$.

\begin{theorem}\label{thm:B_to_D}
    Let $k$ and $q$ be integers with $q\geq2$. Let $c_0,c_1\in\mathbb{C}$ with $c_0\neq c_1$. Let $B=(z_{p,q},\kappa)$, where $\kappa$ is an interpreter function defined by $\kappa((x,y))=c_x\zeta_q^{ky}$. Then $B\sim D_{2,q,k}$.
\end{theorem}

Note that this theorem works only with $p=2$, which unfortunately means that our main results will be restricted to the case where $p=2$ as well.

\begin{proof}

We will show $B\sim D_{2,q,k}$ by showing that $\frac{2}{c_0-c_1}\cdot B(qn)=D_{2,q,k}(qn)$ for all $n\in\mathbb{N}$. First observe that the curve $D_{2,q,k}$ is simply $B$ in the case where $c_0=1$ and $c_1=-1$. Thus, we can calculate the general form of $B(qn)$ to better understand both curves.
\begin{align*}
    B(qn)&=\sum_{j=0}^{qn-1}P_B(z_{2,q}(j))\\
    &=\sum_{(x,y)\in A_{2,q}}\sum_{\substack{0\leq j<qn \\ z_{2,q}(j)=(x,y)}}P_B(z_{2,q}(j))\\
    &=\sum_{(x,y)\in A_{2,q}}|z_{2,q}[:qn]|_{(x,y)}P_B((x,y))\\
    &=\sum_{y=0}^{q-1}\left(|z_{2,q}[:qn]|_{(0,y)}P_B((0,y))+|z_{2,q}[:qn]|_{(1,y)}P_B((1,y))\right)\\
    &=\sum_{y=0}^{q-1}\left(|z_{2,q}[:qn]|_{(0,y)}c_0\zeta_q^{ky}+|z_{2,q}[:qn]|_{(1,y)}c_1\zeta_q^{ky}\right)\\
    &=\sum_{y=0}^{q-1}\left(\zeta_q^{ky}(|z_{2,q}[:qn]|_{(0,y)}c_0+|z_{2,q}[:qn]|_{(1,y)}c_1)\right)
\end{align*}

For any $0\leq y<q$, it is not hard to see that $|f_q[:qn]|_y=n$. We then have 

$$|z_{2,q}[:qn]|_{(0,y)}+|z_{2,q}[:qn]|_{(1,y)}=|f_q[:qn]|_y=n.$$

Thus, in our calculation above, we can substitute $|z_{2,q}[:qn]|_{(1,y)}$ with $n-|z_{2,q}[:qn]|_{(0,y)}$.
\begin{align*}
    B(qn)&=\sum_{y=0}^{q-1}\zeta_q^{ky}(|z_{2,q}[:qn]|_{(y,0)}c_0+(n-|z_{2,q}[:qn]|_{(0,y)})c_1)\\
    &=\sum_{y=0}^{q-1}\zeta_q^{ky}(|z_{2,q}[:qn]|_{(y,0)}c_0+nc_1-|z_{2,q}[:qn]|_{(0,y)}c_1)\\
    &=\sum_{y=0}^{q-1}\zeta_q^{ky}(|z_{2,q}[:qn]|_{(0,y)}(c_0-c_1)+nc_1)\\
    &=(c_0-c_1)\sum_{y=0}^{q-1}\zeta_q^{ky}|z_{2,q}[:qn]|_{(0,y)}+nc_1\sum_{y=0}^{q-1}\zeta_q^{ky}\\
    &=(c_0-c_1)\sum_{y=0}^{q-1}\zeta_q^{ky}|z_{2,q}[:qn]|_{(0,y)}+nc_1(0)\\
    &=(c_0-c_1)\sum_{y=0}^{q-1}\zeta_q^{ky}|z_{2,q}[:qn]|_{(0,y)}
\end{align*}

The second to last equality is due to the fact that $\zeta_q^k$ is a $q$th root of unity. As mentioned earlier, we can now take $c_0=1$ and $c_1=-1$ to obtain $D_{2,q,k}(qn)$.
\begin{align*}
    D_{2,q,k}(qn)&=(1-(-1))\sum_{y=0}^{q-1}\zeta_q^{ky}|z_{2,q}[:qn]|_{(0,y)}\\
    &=2\cdot\sum_{y=0}^{q-1}\zeta_q^{ky}|z_{2,q}[:qn]|_{(0,y)}
\end{align*}

With this, we can now obtain the desired result.
\begin{align*}
    \frac{2}{c_0-c_1}\cdot B(qn)&=\frac{2}{c_0-c_1}\cdot(c_0-c_1)\sum_{y=0}^{q-1}\zeta_q^{ky}|z_{2,q}[:qn]|_{(0,y)}\\
    &=2\cdot\sum_{y=0}^{q-1}\zeta_q^{ky}|z_{2,q}[:qn]|_{(0,y)}\\
    &=D_{2,q,k}(qn).
\end{align*}
    
\end{proof}

Because $\sim$ is a transitive relation, Theorems \ref{thm:T_to_B} and \ref{thm:B_to_D} together give us the ability to relate a Thue-Morse turtle curve $T=(t_2,\tau)$ with $\alpha_T(\phi(0))=\zeta_q^k$ to the Dekking curve $D_{2,q,k}$. When we consider the implications of this for the convergence of $T$, an issue that we will encounter is that $D_{2,q,k}$ is not necessarily a regular Dekking curve, so we will not be able to apply Theorem \ref{thm:dekking_convergence}. For example, if $\alpha_T(\phi(0))=\zeta_6$, as with the curve shown in Figure \ref{fig:koch_convergence}, we see experimentally that $T$ still appears to converge to the Koch curve. However, $q=6$ and $p=2$ are not coprime, so $D_{2,6,1}$ is not regular. The following theorem gives us a way to work around this issue, as it implies that $D_{2,6,1}\sim D_{2,3,1}$, which is regular.

\begin{theorem}\label{thm:D_to_D}
    Let $R=D_{p,q,k_1}$ be a regular Dekking curve. Let $D=D_{p,qp^b,k_2}$ be a Dekking curve, where $b\in\mb{N}$ and $\gcd(qp^b,k_2)=1$. If there exists $d\in\mathbb{N}$ such that $k_1\equiv p^dk_2\pmod{q}$, then $R\sim D$.
\end{theorem}

\begin{proof}

We claim that $D(p^d)R(n)=D(p^{d+b}n)$ for all $n\in\mb{N}$. We first calculate $D(p^d)R(n)$.
\begin{align*}
    D(p^{d+b})R(n)&=P_{D}(z_{p,qp^b}[:p^{d+b}])P_{R}(z_{p,q}[:n])\\
    &=\left(\sum_{r=0}^{p^{d+b}-1}\zeta_p^{t_p(r)}\zeta_{qp^b}^{k_2r}\right)\left(\sum_{h=0}^{n-1}\zeta_p^{t_p(h)}\zeta_q^{k_1h}\right)
\end{align*}

Now we will show that $D(p^{d+b}n)$ equates to the same expression. Expanding $D(p^{d+b}n)$, we find
\begin{align*}
    D(p^{d+b}n)&=P_{D}(z_{p,qp^b}[:p^{d+b}n])\\
    &=\sum_{j=0}^{p^{d+b}n-1}P_{D}(z_{p,qp^b}(j))\\
    &=\sum_{j=0}^{p^{d+b}n-1}\zeta_p^{t_p(j)}\zeta_{qp^b}^{k_2j}.
\end{align*}

We can re-index the sum by letting $j=p^{d+b}h+r$, with $0\leq h<n$ and $0\leq r<p^{d+b}$. Then, we apply Lemma \ref{lem:p^k_reduction} to $t_p(p^{d+b}h+r)$.
\begin{align*}
    D(p^{d+b}n)&=\sum_{r=0}^{p^{d+b}-1}\sum_{h=0}^{n-1}\zeta_p^{t_p(p^{d+b}h+r)}\zeta_{qp^b}^{p^{d+b}h+r}\\
    &=\sum_{r=0}^{p^{d+b}-1}\sum_{h=0}^{n-1}\zeta_p^{t_p(h)+t_p(r)}\zeta_{qp^b}^{p^{d+b}h+r}\\
    &=\sum_{r=0}^{p^{d+b}-1}\sum_{h=0}^{n-1}\zeta_p^{t_p(r)}\zeta_{qp^b}^{k_2r}\zeta_p^{t_p(h)}\zeta_{qp^b}^{p^{d+b}h}\\
    &=\left(\sum_{r=0}^{p^{d+b}-1}\zeta_p^{t_p(r)}\zeta_{qp^b}^{k_2r}\right)\left(\sum_{h=0}^{n-1}\zeta_p^{t_p(h)}\zeta_{qp^b}^{k_2p^{d+b}h}\right)\\
    &=\left(\sum_{r=0}^{p^{d+b}-1}\zeta_p^{t_p(r)}\zeta_{qp^b}^{k_2r}\right)\left(\sum_{h=0}^{n-1}\zeta_p^{t_p(h)}\zeta_q^{k_2p^dh}\right)\\
    &=\left(\sum_{r=0}^{p^{d+b}-1}\zeta_p^{t_p(r)}\zeta_{qp^b}^{k_2r}\right)\left(\sum_{h=0}^{n-1}\zeta_p^{t_p(h)}\zeta_q^{k_1h}\right)\\
    &=D(p^{d+b})R(n)
\end{align*}

The second to last equality is by our initial assumption that $k_1\equiv p^dk_2\pmod{q}$. So because $D(p^d)R(n)=D(p^{d+b}n)$ for all $n\in\mb{N}$, $R\sim D$.

\end{proof}

Now we are ready to put Theorems \ref{thm:T_to_B}, \ref{thm:B_to_D}, and \ref{thm:D_to_D} together to establish our main result.

\begin{theorem}\label{thm:main_result}
    Let $R=D_{2,q,k_1}$ be a regular Dekking curve with scaling factor $r$ which has $q$ odd. Let $T=(t_2,\tau)$ be a Thue-Morse turtle curve such that $\alpha_T(\phi(0))=\zeta_{2^bq}^{k_2}$, where $\phi$ is the morphism used to define $t_2$, and $b,k_2\in\mb{N}$ with $\gcd(k_2,q)=1$. Also require that $P_T(\phi(0))\neq P_T(\phi(1))$. If there exists $d\in\mb{N}$ such that $k_1\equiv 2^dk_2\pmod{q}$, then $R$ and $T$ share a nontrivial limit curve $K$, with
    $$\lim_{n\to\infty}r^{-n}\mc{P}(R[:Q^n])=K,$$
    where $Q=2^{\varphi(q)}$.
\end{theorem}

\begin{proof}
    Let $B=(z_{2,2^bq},\kappa)$ be an absolute turtle curve, where $\kappa$ is an interpreter function with $\kappa((x,y))=P_T(\phi(x))\zeta_{2^bq}^{k_2y}$ for all $(x,y)\in A_{2,2^bq}$. The turtle curve $T$ satisfies the conditions of Theorem \ref{thm:T_to_B}, so $T\sim B$. Because we required that $P_T(\phi(0))\neq P_T(\phi(1))$, the turtle curve $B$ satisfies that conditions of Theorem \ref{thm:B_to_D}, so $B\sim D_{2,2^bq,k_2}$. Using our assumption that there exists $d\in\mb{N}$ such that $k_1\equiv 2^dk_2\pmod{q}$, Theorem \ref{thm:D_to_D} tells us that $D_{2,2^bq,k_2}\sim D_{2,q,k_1}=R$. Then since $\sim$ is transitive, we have $T\sim R$. By Theorem \ref{thm:dekking_convergence}, because $R$ is a regular Dekking curve, it a nontrivial limit curve $K$, with
    $$\lim_{n\to\infty}r^{-n}\mc{P}(R[:Q^n])=K,$$
    where $Q=2^{\varphi(q)}$. Finally, because $T\sim R$, by Theorem \ref{thm:sim_convergence}, we know that $K$ is also a limit curve of $T$.
\end{proof}

If we are specifically interested in showing that certain Thue-Morse turtle curves converge to the Koch curve, then we can obtain the following corollary.

\begin{corollary}
    Let $R=D_{2,3,1}$ be a Dekking curve. Let $T=(t_2,\tau)$ be a Thue-Morse turtle curve such that $\alpha_T(\phi(0))=\zeta_{3\cdot2^b}^{k}$, where $\phi$ is the morphism used to define $t_2$, $b\in\mb{N}$, and $k$ is an integer with $\gcd(k,6)=1$. Also require that $P_T(\phi(0))\neq P_T(\phi(1))$. Then $R$ and $T$ share a nontrivial limit curve $K$, with
    $$\lim_{n\to\infty}3^{-n}\mc{P}(R[:4^n])=K.$$
\end{corollary}

\begin{proof}
    Consider Theorem \ref{thm:main_result} with $q=3$ and $k_1=1$. We have $Q=2^{\varphi(3)}=2^2=4$, and manual calculation shows that $P_D(\lambda((0,0)))=D_{2,3,1}(4)=3$. So because $|3|>1$, $R=D_{2,3,1}$ is a regular Dekking curve. Because $\gcd(k,6)=1$, either $k\equiv1\pmod{3}$ or $k\equiv2\pmod{3}$. If $k\equiv1\pmod{3}$, then $1\equiv2^0k\pmod{3}$. If $k=2$, we have $k\equiv2\pmod{3}$. So either way, there exists $d\in\mb{N}$ such that $1\equiv 2^dk\pmod{3}$. Therefore, all of the conditions of Theorem \ref{thm:main_result} are satisfied, giving the desired result.
\end{proof}

With this corollary, we have confirmed the conjecture made by Zantema in \cite{zantema}: if the sum of angles of a Thue-Morse turtle curve $T$ is of the form $\frac{k\pi}{3\cdot2^n}$ for some $n\in\mb{N}$ and $k$ with $\gcd(k,6)=1$, then $T$ will converge to the Koch curve (with the exception of the case where $P_T(\phi(0))=P_T(\phi(1))$).

\section{Acknowledgments}

This paper is based on research originally conducted through the 2023 Kenyon College Summer Science Scholars program. I would like to thank my mentor, Professor Judy Holdener, for her invaluable guidance throughout the research and writing process.

\bibliography{references}{}
\bibliographystyle{plain}

\end{document}